\documentclass[11pt, oneside]{article}   	
\usepackage{geometry}                		
\usepackage{graphicx}				
								
\DeclareGraphicsExtensions{.eps}	
\usepackage{amssymb,amsmath,amsthm,bbm}
\usepackage{caption}
\usepackage{subcaption}
\usepackage{color}
\usepackage{enumerate}
\usepackage[toc]{appendix}

\newtheorem{theorem}{Theorem}
\newtheorem{lemma}[theorem]{Lemma}

\newtheorem{corollary}[theorem]{Corollary}
\newtheorem{proposition}[theorem]{Proposition}
\newtheorem{definition}[theorem]{Definition}
\newtheorem{claim}[theorem]{Claim}

\newcommand{\Hm}{\operatorname{Hm}}

\newcommand{\dist}{\operatorname{dist}}
\newcommand{\supp}{\operatorname{supp}}
\newcommand{\mc}[1]{\mathcal{#1}}

\newcommand{\mbb}[1]{\mathbb{#1}}
\newcommand{\til}[1]{\tilde{#1}}

\newcommand{\E}{{\mathbb{E}}}

\newcommand{\one}{\mathbf{\mathbbm{1}}}
\newcommand{\0}{\mathbf{0}}
\newcommand{\GFF}{\operatorname{GFF}}
\newcommand{\IV}{\operatorname{IV}}
\newcommand{\Sym}{\operatorname{Sym}}
\newcommand{\Harm}{\operatorname{Harm}}
\newcommand{\sep}{\operatorname{sep}}
\newcommand\numberthis{\addtocounter{equation}{1}\tag{\theequation}}

\usepackage{url}

\title{
Maximum of the integer-valued Gaussian free field
}
\author{Mateo Wirth \thanks{Partially supported by an NSF grant DMS-1757479}\\ University of Pennsylvania
}
\date{\today}

\begin{document}
\thispagestyle{empty}
\maketitle
\linespread{1.2}

\begin{abstract}
	We study the maximum of the integer valued Gaussian free field on a two-dimensional box, and prove that it is of order $\log(L)$ (where $L$ is the size of the box) at high temperature. That is, it is of the same order as the maximum of the discrete Gaussian free field. Our treatment follows closely the recent paper of Karash and Peled \cite{KP17}.
\end{abstract}

\section{Introduction and main result}\label{sec-intro}
In 1972, Berezinskii \cite{Berezinskii71, Berezinskii72} and Kosterlitz and Thouless \cite{KosterlitzThouless72, KosterlitzThouless73} predicted the existence of new types of phase transitions leading to topological phases of matter. In 1981, Fr\"{o}hlich and Spencer \cite{FrohlichSpencer81} proved these predictions mathematically. Among the many implications of the Fr\"{o}hlich-Spencer proof is the delocalization of the integer-valued discrete Gaussian free field at high temperature. Nevertheless, mathematical understanding of the fine properties of this model remains incomplete. In this paper we use the techniques in the Fr\"{o}hlich-Spencer proof to show that the maximum of the integer-valued discrete Gaussian free field in a box of side length $L$ is of order $\log L$. Our presentation follows closely that in the very nice expository paper by Kharash and Peled \cite{KP17} on the Fr\"{o}hlich-Spencer proof.

\subsection{Integer-Valued Discrete Gaussian Free Field}
Let $L > 2$ be an integer, and $\Lambda$ be the graph with vertex set $V(\Lambda) = \{0,1,\dots,L-1\}^2$ and edge set $E(\Lambda)$ given by all pairs $\{(a,b),(c,d)\} \subset V(\Lambda)$ where $|a-c| + |b-d| = 1$. We call such a graph a \emph{square domain} of side-length $L$. We let $\partial \Lambda$ be the following subset of $V(\Lambda)$
\begin{equation}\label{eq-boundary def}
\partial \Lambda = \{(a,b) \in V(\Lambda) \,:\, a \in \{0,L-1\} \text{ or } b \in \{0,L-1\}\},
\end{equation}
and $\Lambda^o = V(\Lambda) \setminus \partial \Lambda$. We call these sets the boundary and interior of $\Lambda$, respectively. For two vertices $j,l \in V(\Lambda)$, we write $j \sim l$ if $\{j,l\} \in E(\Lambda)$. To simplify notation, we will identify $\Lambda$ with $V(\Lambda)$ from now on.

We now introduce the main object of interest, the \emph{integer-valued discrete Gaussian free field} on $\Lambda$. For a function $h: \partial \Lambda \to \mbb{Z}$ and a positive constant $\beta$, we say that a random field $m : \Lambda \to \mbb Z$ is an integer-valued discrete Gaussian free field (IV-GFF) on $\Lambda$ at inverse temperature $\beta$ with boundary condition $h$, and write its law as $\mbb P^{\IV}_{\beta,\Lambda,h}$, if the following holds
\begin{equation}\label{eq-IVGFF pmf}
\mbb{P}^{\IV}_{\beta,\Lambda,h}(m = g) = \frac{1}{Z^{\IV}_{\beta,\Lambda,h'}} \exp\left[-\frac{\beta}{2}\sum_{j \sim l} (g_j - g_l)^2\right] \prod_{j \in \partial \Lambda} \one_{\{h_j\}}(g_j),
\end{equation}
where $\one_{A}$ is the indicator function of the set $A$, $Z^{\IV}_{\beta,\Lambda,h} > 0$ is a normalization constant chosen so $\mbb P^{\IV}_{\beta,\Lambda,h}$ is a probability measure, and in the sum $\sum_{j \sim l}$ there is exactly one term for every edge in $E(\Lambda)$. We write $\E^{\IV}_{\beta,\Lambda,h}$ for the corresponding expectation, and $\mbb{P}^{\IV}_{\beta,\Lambda,\0}$ for the law of the field with $h$ identically equal to $0$.

Our main interest is to prove bounds on the maximum of the IV-GFF in a square domain as the side-length goes to infinity. Our first result in this direction is a lower bound on the order of the maximum of the absolute value of the IV-GFF. We include it as a separate theorem as the proof is simpler than that of the main theorem but uses the same technical ingredients.
\begin{theorem}\label{thm-abs max asymptotic}
There exist constants $c_0, \eta_0, \beta_0 > 0$ and $L_0 > 2$ such that the following holds. Let $\Lambda$ be a square domain of side-length $L \geq L_0$ and $0< \beta < \beta_0$. Then
	\[
	\mbb{P}^{\IV}_{\beta,\Lambda, \0}\left(\max_{j \in \Lambda} |m_j| \geq \frac{c_0}{\sqrt{\beta}} \log(L)\right) \geq 1 - L^{-\eta_0}.
	\]
\end{theorem}
By symmetry, this implies that the maximum of the IV-GFF is of order $\log(L)$ with probability bounded away from zero. The second theorem, which we will obtain as a consequence of some results used in the proof of the first, shows that in fact the maximum of the IV-GFF is of order $\log(L)$ with high probability.
\begin{theorem}\label{thm-max asymptotic no rate}
	Let $\beta_0 > 0$ be as in Theorem~\ref{thm-abs max asymptotic}. For every $\epsilon > 0$, there exist constants $c_1 > 0$ and $L_1 > 2$ such that the following holds. Let $\Lambda$ be a square domain of side-length $L \geq L_1$ and $0 < \beta < \beta_0$. Then
	\[
	\mbb{P}^{\IV}_{\beta,\Lambda, \0}\left(\max_{j \in \Lambda} m_j \geq \frac{c_1}{\sqrt{\beta}} \log(L)\right) \geq 1 - \epsilon.
	\]
\end{theorem}
As mentioned above, the main technical work is in proving Theorem~\ref{thm-abs max asymptotic}. The main ingredient in the proof is a lower bound on the moment generating function of a symmetrized version of the IV-GFF (see Proposition~\ref{prop-IV MGF lower bound} below). This is a slight generalization of a similar bound which was stated in \cite{FrohlichSpencer81} for the IV-GFF with zero boundary condition, to the symmetrized IV-GFF with arbitrary boundary condition. This bound, together with the Markov field property of the IV-GFF (see Lemma~\ref{lm-domain Markov property}), will allow us to prove both theorems by partitioning the domain $\Lambda$ into a large number of sub-domains. The idea of bounding the maximum of a process by repeated ``trials'' has appeared many times before, for example in \cite{MilosPeled15} to prove delocalization in a large class of random surface models in two dimensions.

We note that, as mentioned in \cite[Section 7]{KP17}, it is straightforward to prove (using Proposition~\ref{prop-IV MGF upper bound} below, say) that the maximum of IV-GFF is of order at most $\log L$, so that the bounds in Theorem~\ref{thm-abs max asymptotic} and Theorem~\ref{thm-max asymptotic no rate} are of the right order.

\subsection{Discrete Gaussian free field}
In this section we introduce the discrete Gaussian free field (GFF). This field is closely related to the IV-GFF and will play a key role in our arguments. For a function $h: \partial \Lambda \to \mbb R$ and a positive constant $\beta$, we say a random field $\phi: \Lambda \to \mbb R$ is a \emph{discrete Gaussian free field} on $\Lambda$ at inverse temperature $\beta > 0$ with boundary condition $h$, and write its law as $\mbb P^{\GFF}_{\beta,\Lambda,h}$, if its law has the following density
\begin{equation}\label{eq-GFF density}
d\mbb P^{\GFF}_{\beta,\Lambda,h}(g) = \frac{1}{Z_{\beta,\Lambda,h}} \cdot \exp\left(-\frac{\beta}{2}\sum_{j \sim l} (g_j - g_l)^2\right) \prod_{j \in \partial \Lambda} d\delta_{h_j}(g_j) \prod_{j \in \Lambda^o} dg_j,
\end{equation}
where $dg_j$ is the Lebesgue measure on $\mbb R$, $\delta_{h_j}$ is the Dirac delta measure at $h_j$, and $Z_{\beta,\Lambda,h} > 0$ is a normalization constant. As in the integer-valued case, we write $\E^{\GFF}_{\beta,\Lambda,h}$ for the corresponding expectation, and $\mbb{P}^{\GFF}_{\beta,\Lambda,\0}$ for the law of the field with $h$ identically equal to $0$. That this field is indeed Gaussian can be seen in the case of zero boundary condition by noting that the term in the exponent is a quadratic form of $\{\phi_j \,:\, j \in \Lambda^o\}$, and for general boundary condition $h$ by the fact (proved in the next section) that there exists a deterministic function $\tilde h$ such that $\phi - \tilde h$ is a GFF with zero boundary condition.

Compared to the integer-valued case, the maximum of the real-valued Gaussian free field is very well understood. For instance, Bolthausen, Deuschel and Giacomin \cite{BDG01} established the leading asymptotics for the maximum, Bramson, Ding, and Zeitouni \cite{BDZ16} proved that the centered maximum of the GFF converges in law, and Biskup and Louidor \cite{BiskupLouidor16} proved convergence in law for the extremal process.

Such detailed results seem currently out of reach for the integer-valued GFF as the proofs generally rely on fine properties of Gaussian processes which are not available in this case.

\subsection{Discussion}
As mentioned, our focus is on the IV-GFF with zero boundary condition. This is slightly different from \cite{KP17}, which treats the field with either free or periodic boundary condition. We will introduce these models in Section~\ref{sec-prelim} and make a few comments on how to extend our results to cover them in Section~\ref{sec-thm1 sketch} and Section~\ref{sec-thm2 proof}.

The rest of the paper is structured as follows. In Section~\ref{sec-prelim} we introduce some notation and a few simple facts which will be used in the proofs. In Section~\ref{sec-thm1 sketch} we sketch the proof of Theorem~\ref{thm-abs max asymptotic}, with technical details deferred to subsequent sections. Finally, in Section~\ref{sec-thm2 proof} we prove Theorem~\ref{thm-max asymptotic no rate}.

\section{Notation and preliminaries}\label{sec-prelim}
Before presenting the proofs of the main theorems, we state some simple facts which will be needed. Throughout, we let $\Lambda$ be a square domain, and $\partial \Lambda \subset \Lambda$ be as in \eqref{eq-boundary def}. We write $\dist(j,l)$ for the graph distance between $j$ and $l$ and $\dist(A,B) := \min_{j \in A, l \in B} \dist(j,l)$ for $A,B \subset V(\Lambda)$.

\subsection{Asymptotic notation}
The following notation will be used to describe the asymptotic behavior of functions. For two functions $g_1$ and $g_2$, we say $g_1(L) = O(g_2(L))$ as $L \to \infty$ if there exist constants $c > 0$ and $L_0 > 0$ such that for all $L \geq L_0$, $|g_1(L)| \leq c |g_2(L)|$. We say $g_1(L) = o(g_2(L))$  if for every constant $c > 0$ there exists $L_0 > 0$ such that for all $L \geq L_0$, $|g_1(L)| \leq c |g_2(L)|$. For a collection of functions $\{g_\alpha \,:\, \alpha \in A\}$ indexed by a set $A$ (usually $A = \Lambda$), and another function $g_2$, we say $g_\alpha = O(g_2)$ uniformly in $\alpha$ if there exist constants $c > 0$ and $L_0 > 0$ such that for all $L \geq L_0$ and $\alpha \in A$, $|g_\alpha(L)| \leq c |g_2(L)|$. We say $g_\alpha = o(g_2)$ uniformly in $\alpha$ if the analogous condition holds.

\subsection{The symmetrized field}
 In this section we introduce symmetrized versions of the GFF and the IV-GFF. Briefly, these are obtained by multiplying the appropriate field by an independent, unbiased random sign. More concretely, we call a random integer-valued field $n: \Lambda \to \mbb Z$ a \emph{symmetrized integer-valued discrete Gaussian free field} with boundary condition $h$ at inverse temperature $\beta > 0$, and write its law as $\mbb P^{\IV-\Sym}_{\beta,\Lambda,h}$, if there exists an IV-GFF $m$ (with appropriate parameters) and an independent mean-zero random variable $X$ taking values in $\{-1,1\}$ such that $n = X\cdot m$. We note that
\[
\mbb P^{\IV-\Sym}_{\beta,\Lambda,h}(\cdot) = \frac{1}{2}\left[\mbb P^{\IV}_{\beta,\Lambda,h}(\cdot) + \mbb P^{\IV}_{\beta,\Lambda,-h}(\cdot)\right],
\]
and in particular $\mbb P^{\IV-\Sym}_{\beta,\Lambda,\0} = \mbb P^{\IV}_{\beta,\Lambda,\0}$. The symmetrized discrete Gaussian free field is defined analogously. These symmetrized fields are used to prove a lower bound on the moment generating function of the integer-valued field with non-zero boundary condition. This constitutes a minor extension of the main bound in the Fr\"{o}hlich-Spencer proof.

\subsection{Harmonic functions}
We define the Laplacian $\Delta_\Lambda$ as the linear operator satisfying
\begin{equation}\label{eq-LaplacianDef}
(\Delta_\Lambda f)_j := \sum_{l : l \sim j} (f_l - f_j), \quad f : \Lambda \to \mbb R.
\end{equation}
We say a function $\tilde h : \Lambda \to \mbb{R}$ is harmonic in $\Lambda^o$ (or simply harmonic) if
\[
(\Delta_\Lambda \tilde h)_j = 0, \quad j \in \Lambda^o.
\]
We denote the space of harmonic functions by $\Harm(\Lambda)$. We note that for any $h:\partial \Lambda \to \mbb R$ there is a unique function $\tilde h \in \Harm(\Lambda)$, which we call the harmonic extension of $h$, such that $\tilde h_j = h_j$ for $j \in \partial \Lambda$. In fact, $\tilde h$ can be constructed as follows. Let $S$ be a continuous-time simple random walk on $\Lambda$ with transition rate 1 from $j$ to $l$, whenever $j \sim l$, and $\zeta = \inf \{t \geq 0 \,:\, S_t \in \partial \Lambda\}$. We let $\mbb P_j$ be the law of $S$ with $S_0 = j$, and $\E_j$ denote the expectation with respect to $\mbb P_j$. We can then define the function $\Hm_{\Lambda} :\Lambda \times \partial \Lambda \to \mbb R$, which we call the harmonic measure on $\partial\Lambda$, by $\Hm_{\Lambda}(j,l) := \mbb P_j(S_\zeta = l)$. Finally, we let
\begin{equation}\label{eq-Harmonic extension}
\tilde h_j = \sum_{l \in \partial \Lambda} \Hm_{\Lambda}(j,l)h_l.
\end{equation}
We note that if $\phi$ is a GFF with boundary condition $h$, then $\phi-\tilde h$ is a GFF with zero boundary condition. This follows directly from \eqref{eq-GFF density}. Since the law of a GFF with zero boundary condition is invariant under the mapping $\phi \to -\phi$, it follows that  $\E^{\GFF}_{\beta,\Lambda,h}[\phi] = \tilde h$. For ease of notation, we will identify a function $h$ on $\partial \Lambda$ with its harmonic extension $\tilde h$ whenever there is no risk of confusion. 

\subsection{The Green's function}
In this section we introduce the Green's function of the simple random walk (killed on $\partial \Lambda$). As above, we let $S$ be a simple random walk on $\Lambda$ and $\zeta$ be the hitting time of $\partial \Lambda$ by $S$. We define the Green's function $G_\Lambda: \Lambda^2 \to \mbb{R}$ by
\begin{equation}\label{eq-Greenfunc}
G_\Lambda(j,l) := \E_j\left[ \int_0^\zeta \one_{\{l\}}(S_t) dt\right].
\end{equation}
The importance of $G_\Lambda$ for our arguments comes from the fact that it is proportional to the covariance of the GFF and is therefore closely related to its moment generating function. This will follow from the following fact.
\begin{claim}\label{claim-Green's Laplacian}
Let $l \in \Lambda^o$ be a vertex and $\sigma^l: \Lambda \to \mbb R$ be given by $\sigma^l_j = G_\Lambda(j,l)$. Then
\[
(\Delta_\Lambda \sigma^l)_j =
	\begin{cases}
	-1 &j = l,\\
	\Hm(l,j) &j \in \partial \Lambda,\\
	0 &\text{otherwise}.
	\end{cases}
\]
\end{claim}
\begin{proof}
We begin by noting that for $j \in \Lambda^o \setminus\{l\}$, we have by the strong Markov property for simple random walk that
\[
G_\Lambda(j,l) = \frac{1}{4}\sum_{k : k \sim j} G_\Lambda(k,l),
\]
where we have used the fact that $j$ has degree 4 (i.e. is adjacent to 4 vertices). It follows that $(\Delta_{\Lambda} \sigma^l)_j = 0$. Similarly, noting that the expected time it takes $S$ to jump in $\Lambda^o$ is $\frac{1}{4}$, we have
\[
\sigma^l_l = G_{\Lambda} (l,l) = \frac{1}{4} + \frac{1}{4}\sum_{j:j \sim l} G_{\Lambda}(j,l).
\]
It follows that $(\Delta_{\Lambda} \sigma^l)_l = -1$. For $j \in \partial \Lambda$, we want to show
\[
(\Delta_{\Lambda}\sigma^l)_j = \sum_{k: k \sim j} G_{\Lambda}(k,l) = \Hm(l,j).
\]
This follows from a last exit decomposition of the event $\{S_\zeta = j\}$. To see this, we let $\hat S$ be the discrete time simple random walk associated to $S$ and $\hat G_{\Lambda}$ be its Green's function. That is, for $j,l \in \Lambda$
\begin{equation}\label{eq-discrete Green's function}
\hat G_{\Lambda}(j,l) = \mathbb E_j\left[ \sum_{n = 0}^{\hat \zeta-1} \one_{\{l\}}(\hat S_n)\right] = 4 G(j,l),
\end{equation}
where $\hat \zeta = \min \{n \geq 0 \,:\, \hat S \in \partial \Lambda\}$. By \cite[Lemma 6.3.6]{LawlerLimic10}, we have for $j \in \partial \Lambda$, $l \in \Lambda^o$
\[
\Hm(l,j) = \sum_{k: k \sim j} \hat G_{\Lambda}(l,k) \frac{1}{4} = \sum_{k: k \sim j} G_\Lambda(l,k).
\]
To conclude, we note that $G_\Lambda$ is symmetric, so $G_\Lambda(l,k) = G_\Lambda(k,l)$. This follows from the fact that $\hat G_\Lambda$ is symmetric, which is an easy consequence of the symmetry of the simple random walk on $\Lambda^o$ (see \cite[Lemma 4.6.1]{LawlerLimic10}).
\end{proof}

As mentioned above, $G_\Lambda$ is closely related to the moment generating function of the GFF. Thus, it will be useful to characterize its asymptotic behavior as the side-length of the domain grows. It follows from \cite[Theorem 4.4.4, Proposition 4.6.2]{LawlerLimic10} that the following holds uniformly over $\Lambda$ and $j \in \Lambda^o$
\begin{equation}\label{eq-Green asymptotic}
G_{\Lambda}(j,j) = \frac{1}{2\pi}\log(\dist(j, \partial \Lambda)) + O(1), \quad \text{as} \quad \dist(j,\partial \Lambda) \to \infty.
\end{equation}
We note that the results in \cite{LawlerLimic10} are stated for $\hat G_{\Lambda}$ which differs from $G$ by a factor of $4$. This accounts for the discrepancy between \eqref{eq-Green asymptotic} above and the corresponding statements in \cite{LawlerLimic10}.

\subsection{The orthogonal complement of $\Harm(\Lambda)$}\label{sec-harm perp}
For $f,g: \Lambda \to \mbb{R}$ we write
\begin{align*}
\langle f, g \rangle &:= \sum_{j \in \Lambda} f_jg_j.
\end{align*}
We denote by $\Harm(\Lambda)^{\perp}$ the space of functions $f:\Lambda \to \mbb R$ such that $\langle f,\tilde h\rangle = 0$ for all $\tilde h \in \Harm(\Lambda)$. For the rest of the paper, we always take $f$ to be an element of $\Harm(\Lambda)^\perp$.
\begin{claim}\label{claim-harm perp}
A function $f : \Lambda \to \mbb R$ is in $\Harm(\Lambda)^\perp$ if and only if there exists a function $\sigma:\Lambda \to \mbb R$ such that $-\Delta_{\Lambda} \sigma = f$ and $\sigma_j = 0$ for all $j \in \partial \Lambda$.
\end{claim}
We note that the function $\sigma$ in the claim is necessarily unique since $\Delta_\Lambda \sigma = \Delta_\Lambda g$ if and only if $\sigma - g$ is a constant function. Therefore, we will write $\sigma = -(\Delta_\Lambda)^{-1} f$. Throughout the rest of the paper, $\sigma$ will denote such a function.

We note that for any function $g$ on $\Lambda$, we have
\[
\sum_{j \sim l} (g_j - g_l)^2 = \langle g, -\Delta_{\Lambda} g\rangle.
\]
Therefore, assuming the claim holds, the following change of variables $\phi \to \phi+\beta^{-1}\sigma$ shows that for all (bounded, measurable) functions $g$ of $\phi$ we have
\begin{align*}
\E^{\GFF}_{\beta,\Lambda,h}[e^{\langle\phi, f \rangle}g(\phi)] &= \frac{1}{Z^{\GFF}_{\beta,\Lambda,h}}\cdot\int_{\mbb R^\Lambda} g(\phi) e^{\langle \phi, -\Delta_\Lambda \sigma\rangle-\frac{\beta}{2} \langle \phi,-\Delta_\Lambda \phi\rangle} \prod_{j \in \partial \Lambda}d\delta_{h_j}(\phi_j) \prod_{j \in \Lambda^o} d\phi_j\\
&= \frac{\exp\left(\frac{1}{2\beta} \langle \sigma, f\rangle \right)}{Z^{\GFF}_{\beta,\Lambda,h}}\cdot \int_{\mbb R^\Lambda} g\left(\phi + \frac{\sigma}{\beta}\right)e^{-\frac{\beta}{2} \langle \phi,-\Delta_\Lambda \phi\rangle} \prod_{j \in \partial \Lambda}d\delta_{h_j}(\phi_j) \prod_{j \in \Lambda^o} d\phi_j\\ 
&= \exp\left(\frac{1}{2\beta} \langle \sigma, f\rangle \right)\E^{\GFF}_{\beta,\Lambda,h}\left[g\left(\phi + \frac{\sigma}{\beta}\right)\right]. \numberthis \label{eq-GFF MGF}
\end{align*}
It follows that the same holds for the symmetrized field.

\begin{proof}[Proof of Claim~\ref{claim-harm perp}]
It is easy to see that if $f = -\Delta_\Lambda \sigma$, $\langle f, h\rangle = \langle \sigma , -\Delta_\Lambda h \rangle = 0$ since $\Delta_\Lambda h$ vanishes on $\Lambda^o$. To prove the converse, we suppose $f \in \Harm(\Lambda)^{\perp}$ and construct $\sigma$. Let $l \in \partial \Lambda$ be a vertex on the boundary and $\til h \in \Harm(\Lambda)$ be given by $\til h_j = \Hm_{\Lambda}(j,l)$. We observe
\[
0 = \langle f,h \rangle = f_l + \sum_{j \in \Lambda^o} \Hm_{\Lambda}(j,l)f_j.
\]
That is, 
\begin{equation}\label{eq-Boundary value of antiharmonic function}
f_l = - \sum_{j \in \Lambda^o} \Hm_{\Lambda}(j,l)f_j.
\end{equation}
For $j \in \Lambda^o$, let $f^j$ be the following function
\[
f^j_k := \begin{cases}
	1 &  k= j,\\
	-\Hm_{\Lambda}(j,k) &k \in \partial \Lambda,\\
	0 &\text{otherwise}.
\end{cases}
\]
It follows from \eqref{eq-Harmonic extension} that $f^j \in \Harm(\Lambda)^\perp$ so by \eqref{eq-Boundary value of antiharmonic function}, $\{f^j \,:\, j \in \Lambda^o \}$ is a basis for $\Harm(\Lambda)^\perp$. Finally, it follows from Claim~\ref{claim-Green's Laplacian} that for any $j \in \Lambda^o$ there exists a function $\sigma^j : \Lambda \to \mbb R$ that vanishes on $\partial \Lambda$ such that $f^j  = -\Delta_\Lambda \sigma^j$.
\end{proof}

\subsection{Square sub-domains}
For a square domain $\Lambda$, we call $\Pi \subset \Lambda$ a square sub-domain of side-length $R$ if there exist $(a,b) \in \Lambda$ and an integer $R \geq 1$ such that
\[
\Pi = \{(c,d) \in \Lambda \,:\, (c-a,d-b) \in \{0,1,\dots,R-1\}^2\},
\]
where we assume that $R$ and $(a,b)$ are such that $(a+R-1,b+R-1) \in \Lambda$.

\subsection{Markov field property}
The following Markov field property of the IV-GFF follows directly from \eqref{eq-IVGFF pmf}.
\begin{lemma}\label{lm-domain Markov property}
	Let $\Lambda$ be a square domain, and $m$ be an IV-GFF on $\Lambda$ at inverse temperature $\beta > 0$ with boundary condition $h: \partial \Lambda \to \mbb Z$. Let $\Pi \subset \Lambda$ be a square sub-domain of $\Lambda$. Then $\{m_j \,:\, j \in \Pi\}$ is conditionally independent of $\{m_j \,:\, j \in \Lambda \setminus \Pi\}$ given $\{m_j \,:\, j \in \partial \Pi\}$. Additionally, for any function $h': \partial \Pi \to \mbb Z$, the conditional distribution of $\{m_j \,:\, j \in \Pi\}$ given $m_j = h'_j$ for all $j \in \partial \Pi$ is that of an IV-GFF on $\Pi$ at inverse temperature $\beta$ with boundary condition $h'$.
\end{lemma}

\subsection{Free and periodic boundary conditions}
Here we introduce two slightly different versions of the IV-GFF. We begin by describing the IV-GFF with \emph{free boundary condition}. For a square domain $\Lambda$, $\beta > 0$, and $v \in \Lambda$, we say a random field $m : \Lambda \to \mbb Z$ is an IV-GFF with free boundary condition at inverse temperature $\beta$, and write its law as $\mbb P^{\IV}_{\beta,\Lambda,v}$, if it satisfies
\[
\mbb P^{\IV}_{\beta,\Lambda,v}(m = g) = \frac{1}{Z^{\IV}_{\beta,\Lambda,v}} \cdot \exp\left[-\frac{\beta}{2}\sum_{j \sim l}(g_j - g_l)^2\right] \one_{\{0\}}(g_v).
\]
Similarly, we write $\mbb P^{\GFF}_{\beta,\Lambda,v}$ for the probability measure on functions $\phi: \Lambda \to \mbb R$ with density
\[
d\mbb P^{\GFF}_{\beta,\Lambda,v}(\phi) = \frac{1}{Z^{\GFF}_{\beta,\Lambda,v}} \cdot \exp\left[-\frac{\beta}{2}\sum_{j \sim l}(g_j - g_l)^2\right] \one_{[-\pi,\pi)}(\phi_v) \prod_{j \in \Lambda}d\phi_j,
\]
where as before $d\phi_j$ is the Lebesgue measure. This is the GFF with free boundary condition. We choose the normalization $\phi_v \in [-\pi,\pi)$ instead of the more common $\phi_v = 0$ for consistency with \cite{KP17}. This amounts to adding an independent random variable, uniform on $[-\pi,\pi)$, at every point in $\Lambda$ to a field with the usual normalization. This doesn't affect the order of the maximum.

The field with periodic boundary condition will be denoted by the same notation ($\mbb P^{\IV}_{\beta,\Lambda,v}$ and $\mbb P^{\GFF}_{\beta,\Lambda,v}$) and the formulas for the densities stated above remain valid. The difference is that we alter the graph $\Lambda$ slightly by making it a discrete torus. That is we take $E(\Lambda)$ to be the set of all pairs $(a,b),(c,d) \in V(\Lambda)$ with $(a,b)$ and $(c,d)$ equal in one coordinate and differing by one modulo $L$ in the other. Additionally, we require $L$ to be even so that the graph is bipartite.

Both the free boundary and periodic boundary fields have the Markov field property stated in Lemma~\ref{lm-domain Markov property}, in the sense that for a square sub-domain $\Pi$, given $\{m_j \,:\, j \in \partial \Pi\}$, $\{m_j \,:\, j \in \Pi\}$ is independent of $\{m_j \,:\, j \in \Lambda \setminus \Pi\}$ and is an IV-GFF on $\Pi$ with boundary condition given by $\{m_j \,:\, j \in \partial \Pi\}$.


\section{Proof of Theorem~\ref{thm-abs max asymptotic} - overview}\label{sec-thm1 sketch}
\subsection{Domain decomposition}
The main step in the proof is establishing the following result.
\begin{proposition}\label{prop-maximum epsilon proba}
There exist constants $D_1, R_0, \beta_0 > 0$ such that the following holds. Let $0 < \beta < \beta_0$, $\Lambda$ be a square domain of side-length $L$, and $\Pi \subset \Lambda$ be a square sub-domain of side-length $R \geq R_0$. Let
\[
U = \mbb P^{\IV}_{\beta,\Lambda,\0}\left(  \max_{j \in \Pi} |m_j| \geq \frac{\log(R)}{\sqrt{\beta}} \mid m_j,\, j \in \partial \Pi \right).
\]
Then for any $w > D_1$,
\[
\mbb P^{\IV}_{\beta,\Lambda,\0} \left( U < R^{-w} \right) \leq R^{-\frac{w^2}{ D_1}}.
\]
\end{proposition} 
Before proving the proposition, we show how it implies Theorem~\ref{thm-abs max asymptotic}. We assume, without loss of generality, that $D_1 \geq 10$. Let $\gamma = D_1^{-1}$ and note that $2(1-\gamma) - 3/2 \geq 3/10$. Let $R = \lfloor L^\gamma \rfloor$, $\Pi_{0,0} = \{0,\dots R-1\}^2$, and, for integers $x,y$, $\Pi_{x,y} = \Pi_{0,0} + R\cdot(x,y)$. Let $Q = \{\Pi_{x,y} \,:\, 0 \leq x,y < L^{1-\gamma}\}$ and
\[
\partial Q := \bigcup_{0 \leq x,y < L^{1-\gamma}} \partial \Pi_{x,y}.
\]
Note that $Q$ is a collection of disjoint square sub-domains of $\Lambda$ of side-length $R$. By Lemma~\ref{lm-domain Markov property}
\begin{align*}
U_{x,y} &:= \mbb P^{\IV}_{\beta,\Lambda,\0}\left(  \max_{j \in \Pi} |m_j| \geq \frac{\log(R)}{\sqrt{\beta}} \mid m_j,\, j \in \partial \Pi_{x,y} \right) \\
	&= \mbb P^{\IV}_{\beta,\Lambda,\0}\left(  \max_{j \in \Pi} |m_j| \geq \frac{\log(R)}{\sqrt{\beta}} \mid m_j,\, j \in \partial Q \right).
\end{align*}
Let  $\mc E$ be the following event
\[
\mc E = \left\{\min_{0 \leq x,y < L^{1-\gamma}}U_{x,y} \geq R^{- 3 D_1/2} \right\}.
\]
By Lemma~\ref{lm-domain Markov property}, given $\{m_j \,:\, j \in \partial Q\}$ the restriction of the field $m$ to $\Pi_{x,y}$ is independent of the field outside $\Pi_{x,y}$. Therefore, the following holds almost surely on $\mc E$
\begin{align*}
\mbb P^{\IV}_{\beta,\Lambda,\0}\left(\max_{j \in \Lambda}|m_j| < \frac{\log(R)}{\sqrt{\beta}}\mid m_j, \, j \in \partial Q\right) &\leq (1 - R^{- 3D_1/2})^{|Q|}.
\end{align*}
Using the fact that $R \leq L^\gamma$ and $|Q| \geq L^{2(1-\gamma)}$, we conclude that
\begin{align*}
\mbb P^{\IV}_{\beta,\Lambda,\0}\left(\max_{j \in \Lambda}|m_j| < \frac{\log(R)}{\sqrt{\beta}}\mid m_j, \, j \in \partial Q\right) &\leq \exp\left(-L^{3/10}\right),
\end{align*}
almost surely on $\mc E$. Next, we show $\mc E$ occurs with high probability. Assuming $0 < \beta < \beta_0$ and $L \geq L_1 : = R_0^{1/\gamma}$, we can apply Proposition~\ref{prop-maximum epsilon proba} with $w = 3D_1/2$ and use a union bound to conclude
\[
\mbb{P}^{\IV}_{\beta,\Lambda,\0}\left(\mc E^c\right) \leq |Q| R^{-9D_1/4}.
\]
Assuming without loss of generality that $R_0 \geq 10$, we have $R \geq L^{9\gamma/10}$. Noting the trivial bound $|Q| \leq L^2$, we conclude that
\[
\mbb{P}^{\IV}_{\beta,\Lambda,\0}\left(\mc E^c\right) \leq L^{-1/40}.
\]
This concludes the proof of Theorem~\ref{thm-abs max asymptotic}. We turn now to the proof of Proposition~\ref{prop-maximum epsilon proba}.

\subsection{Proof of Proposition~\ref{prop-maximum epsilon proba}}
The proof consists of using upper and lower bounds on the moment generating function of $m$ to establish a lower bound for the tail of its distribution. The upper bound is given by the following result.
\begin{proposition}[{\cite[Proposition 1.2]{KP17}}]\label{prop-IV MGF upper bound}
	Let $\Lambda$ be a square domain, $\beta > 0$, $f \in \Harm(\Lambda)^\perp$, and $\sigma = -(\Delta_\Lambda)^{-1}f$. Then
	\[
	\E^{\IV}_{\beta,\Lambda,\0}[e^{\langle m, f \rangle}] \leq \exp\left(\frac{1}{2\beta} \langle \sigma, f\rangle\right).
	\]
\end{proposition}
We note that \cite[Proposition 1.2]{KP17} is stated for the field with free or periodic boundary conditions but the proof applies to the zero boundary case as well.

To state the lower bound, we introduce some notation. For a square sub-domain $\Lambda' \subset \Lambda$ we will let, by a slight abuse of notation,  $\Harm(\Lambda')^\perp \subset \Harm(\Lambda)$ be the set of functions $f:\Lambda \to \mbb R$ such that $\sigma = -(\Delta_\Lambda)^{-1}f$ satisfies $\sigma_j = 0 $ for $j \notin (\Lambda')^o$. With this notation, we have the following result.
\begin{proposition}\label{prop-IV MGF lower bound}
For any $\epsilon > 0$ there exists $\beta_0 > 0$ such that the following holds. Let $\Lambda$ be a square domain of side-length $L$, $ 0 < \beta < \beta_0$, and $h: \partial \Lambda \to \mbb Z$. Let $\Lambda' \subset \Lambda$ be a square sub-domain such that $\dist(\Lambda',\partial\Lambda) \geq L/8$. Then for all $f \in \Harm(\Lambda')^\perp$,
	\[
	\E^{\IV-\Sym}_{\beta,\Lambda,h}[e^{\langle n, f \rangle}] \geq \exp\left(\frac{1}{2(1+\epsilon)\beta} \langle \sigma, f\rangle\right),
	\]
	where $\sigma= -(\Delta_\Lambda)^{-1}f$.
\end{proposition}
The proof of Proposition~\ref{prop-IV MGF lower bound} is very similar to that of \cite[Theorem 1.1]{KP17}, but we specify the necessary adjustments in Section~\ref{sec-thm1 details} and the Appendix. We also note that this is the place in the proof where we use the symmetrized field.

We are now ready to give the proof of Proposition~\ref{prop-maximum epsilon proba}. Let $\Pi = \{0,1\dots,R-1\}^2 + (a,b)$ and $j^* = (c,d)$ be the vertex such that
\[
c - a = d-b = \left\lfloor \frac{R-1}{2}\right\rfloor.
\] 
We say $j^*$ is the center of $\Pi$. Next, we let
\[
R' := \left\lfloor \frac{R}{2}\right\rfloor,
\]
and $\Pi' \subset \Pi$ be a square subdomain of side length $R'$ such that $j^*$ is the center of $\Pi'$ in the same sense. That is, we choose $(a',b') \in \Pi$ such that
\[
c - a' = d- b' = \left\lfloor \frac{R'-1}{2}\right\rfloor,
\]
and let $\Pi' = \{0,1,\dots,R'-1\}^2 + (a',b')$. We assume from now on that $R \geq 10$, which implies in particular that $R' \geq 3$ and $\dist(\Pi',\partial\Pi) \geq R/8$. We then take $f^* \in \Harm(\Pi')^\perp$ to be the following function
\[
f^*_j = \begin{cases}
	1 &j = j^*,\\
	-\Hm_{\Pi'}(j^*,j) &j \in \partial \Pi',\\
	0 &\text{otherwise}.
	\end{cases}
\]
We will use Proposition~\ref{prop-IV MGF lower bound} and Proposition~\ref{prop-IV MGF upper bound} to lower bound the probability that $|\langle m, f^*\rangle|$ is of order $\log(R)$. Before proceeding, we note that for any $v \geq 0$, $\langle n, f^*\rangle \geq 2v$ implies that either $n_{j^*} \geq v$ or there exists $j \in \partial \Pi'$ such that $n_j \leq - v$. Thus for any boundary condition $h: \partial \Pi \to \mbb Z$
\begin{align*}
\mbb{P}^{\IV}_{\beta,\Lambda,h}\left(\max_{j \in \Pi} |m_j| \geq \frac{\log(R)}{\sqrt{\beta}}\right) &= \mbb{P}^{\IV-\Sym}_{\beta,\Lambda,h}\left(\max_{j \in \Pi} |n_j| \geq \frac{\log(R)}{\sqrt{\beta}}\right)\\
	 &\geq \mbb{P}^{\IV-\Sym}_{\beta,\Lambda,h}\left(\langle n, f^* \rangle \geq \frac{2\log(R)}{\sqrt{\beta}}\right).
\end{align*}
Thus, it suffices to bound the last term in the last display from below  for ``typical'' boundary conditions. To this end, recall from Section~\ref{sec-harm perp} that $\sigma^*= - (\Delta_\Lambda)^{-1}f^*$ is given by $\sigma^*_j = G_{\Pi'}(j,j^*)$ and
\begin{equation}\label{eq-sub-domain variance}
\langle \sigma^*, f^*\rangle = G_{\Pi'}(j^*,j^*) = \frac{1}{2\pi}\log(R) + O(1), \quad \text{as} \quad R \to \infty.
\end{equation}
where the second equality follows from \eqref{eq-Green asymptotic}. Let $\beta_0 > 0$ be such that Proposition~\ref{prop-IV MGF lower bound} holds with $\epsilon = 1$. We assume from now on that $0 < \beta < \beta_0$. To simplify notation, we let 
\[
\nu = \nu(\beta,R)= \frac{1}{\sqrt{\beta}} \log(R),
\]
and $V = \langle n, f^* \rangle$. We have by Proposition~\ref{prop-IV MGF lower bound} and the Cauchy-Schwarz inequality that for any $s \geq 0$,
\begin{align*}
\exp\left(\frac{s^2}{4\beta}G_{\Pi'}(j^*,j^*)\right) &\leq \E^{\IV-\Sym}_{\beta,\Pi,h}\left[e^{sV}\right]\\
&= \E^{\IV-\Sym}_{\beta,\Pi,h}\left[e^{sV} \one_{(-\infty,2\nu]}(V)\right] + \E^{\IV-\Sym}_{\beta,\Pi,h}\left[e^{sV} \one_{(2\nu,\infty)}(V)\right]\\
&\leq e^{2\nu s} + \E^{\IV-\Sym}_{\beta,\Pi,h}\left[e^{2sV}\right]^{1/2} \mbb{P}^{\IV-\Sym}_{\beta,\Pi,h}(V \geq 2\nu)^{1/2}.
\end{align*}
Let $s_{\beta}$ be the following number
\[
s_{\beta} := 32\pi \sqrt{\beta}.
\]
Recalling \eqref{eq-sub-domain variance}, we see that there exists $R_1$ such that for all $R \geq R_1$ and $0 < \beta < \beta_0$,
\[
\exp\left(\frac{s_\beta^2}{4\beta}G_{\Pi'}(j^*,j^*)\right) - \exp(2 \nu s_\beta) \geq 1.
\]
We assume that this holds from now on. We then have for all boundary conditions $h$
\[
\mbb{P}^{\IV-\Sym}_{\beta,\Pi,h}(V \geq 2\nu) \geq \left(\E^{\IV-\Sym}_{\beta,\Pi,h}\left[e^{2s_\beta V}\right]\right)^{-1}.
\]
To conclude, we need an upper bound on the the expected value of $e^{2s_\beta V}$ that holds with high probability when the boundary condition is sampled from an IV-GFF on $\Lambda$ at inverse temperature $\beta$ with zero boundary condition. This is given by the following lemma
\begin{lemma}\label{lm-conditional MGF upper bound}
There exist positive constants $R_2$ and $D_1$ such that the following holds. Let $\Lambda$, $\beta$, $\Pi$, and $f^*$ be as above. Let $W_+$, $W_-$ be the following random variables
\[
W_{\pm} = \E^{\IV}_{\beta,\Lambda,\0}\left[ e^{\pm 2s_\beta \langle m, f^* \rangle} \mid m_j,\, j \in \partial \Pi \right].
\]
If $R \geq R_2$, then for any $w > D_1$,
\[
\mbb P^{\IV}_{\beta,\Lambda,\0}\left( \max(W_+,W_-) > R^{w}\right) \leq R^{-\frac{w^2}{D_1}}.
\]
\end{lemma}
Provided with the lemma, we conclude the proof of Proposition~\ref{prop-maximum epsilon proba} as follows. Assume $R \geq R_0 : = \max(R_1,R_2)$ and let $h: \partial \Pi \to \mbb Z$ be given by $h_j = m_j$ for all $j \in \partial \Pi$. Note that
\[
\E^{\IV-\Sym}_{\beta,\Pi,h}\left[e^{2s_\beta V}\right] = \frac{W_+ + W_-}{2},
\]
which gives 
\[
U \geq \left(E^{\IV-\Sym}_{\beta,\Pi,h}\left[e^{2s_\beta V}\right]\right)^{-1} \geq \max(W_+,W_-)^{-1}.
\]
It follows immediately that
\[
\mbb P^{\IV}_{\beta,\Lambda,\0}( U < R^{-w}) \leq \mbb P^{\IV}_{\beta,\Lambda,\0}\left( \max(W_+,W_-) > R^w\right) \leq R^{-\frac{w^2}{D_1}}\,
\] 
as required.
\begin{proof}[Proof of Lemma~\ref{lm-conditional MGF upper bound}]
The proof is an application of Markov's inequality. By Jensen's inequality for conditional expectations and Proposition~\ref{prop-IV MGF upper bound}, we have for any $p \geq 1$
\[
\E^{\IV}_{\beta,\Lambda,\0}[W_+^p] \leq \E^{\IV}_{\beta,\Lambda,\0}[e^{2ps_\beta V}] \leq \exp\left(\frac{p^2s^2_\beta}{2\beta}G_{\Pi'}(j^*,j^*)\right).
\]
Therefore,
\begin{align*}
\mbb P^{\IV}_{\beta,\Lambda,\0}\left(W_+ \geq e^x\right) &\leq \inf_{p \geq 1} e^{-px}\cdot \E^{\IV}_{\beta,\Lambda,\0}[W_+^p]\\
&\leq\exp\left(\inf_{p \geq 1} \frac{p^2 s^2_\beta}{2\beta}G_{\Pi'}(j^*,j^*) - px\right) \\
&= \exp\left( - \frac{\beta x^2}{2s^2_\beta G_{\Pi'}(j^*,j^*)}\right), \quad x \geq \frac{s^2_\beta}{\beta}G_{\Pi'}(j^*,j^*)\,.
\end{align*}
By \eqref{eq-sub-domain variance} and the definition $s_\beta = 2^5 \pi \sqrt{\beta}$, there exists $R_2$ such that for $R \geq R_2$
\[
\frac{s^2_\beta}{\beta} G_{\Pi'}(j^*,j^*) = 2^{9}\pi\log(R) + O(1) \leq 2^{10}\pi \log(R).
\]
Therefore, if we let $D_1 = 2^{12}\pi$, we see that for $R \geq R_2$ and $w > D_1$
\[
\mbb P^{\IV}_{\beta,\Lambda,\0}\left(W_+ \geq R^w\right) \leq \exp\left( - \frac{2w^2\log(R)}{D_1}\right)= R^{-\frac{2w^2}{D_1}}\,.
\]
The same bound holds for $W_-$ by symmetry, so a union bound concludes the proof.
\end{proof}
\subsection{Proof of Theorem~\ref{thm-abs max asymptotic} for free and periodic boundary condition}
The argument in this section applies almost without changes to the field with free or periodic boundary condition. Specifically, let $m$ have law $\mbb P^{\IV}_{\beta,\Lambda,v}$ and $Q$ and $\gamma$ be as above. Since the Markov field property still holds it remains true that the variables $\{U_{x,y} \,;\, 0 \leq x,y < L^{1-\gamma}\}$ are conditionally independent given $\{m_j \,:\, j \in \partial Q\}$. Moreover, for $0 \leq x,y < L^{1-\gamma}$ the restriction $m$ to $\Pi_{x,y}$ is distributed as an IV-GFF with boundary condition given by $\{m_j \,:\, j \in \partial \Pi_{x,y}\}$, except possibly for one value of the pair $(x,y)$ such that $v \in (\Pi^{x,y})^{o}$. The argument proceeds by ignoring this sub-domain.

\small

\section{Proof of Proposition~\ref{prop-IV MGF lower bound}}\label{sec-thm1 details}
\subsection{Discrete Gaussian Free Field with Periodic Single-Site Weights}
To prove Proposition~\ref{prop-IV MGF lower bound}, we approximate the integer-valued GFF by a \emph{discrete Gaussian free field with periodic single-site weights}, which we now introduce. We say that $\lambda : \mbb{R} \to \mbb{R}$ is a \emph{real, even, normalized trigonometric polynomial} if
\[
\lambda(x) = 1 + 2\sum_{q = 1}^N\hat\lambda_q \cos(qx),
\]
for some integer $N > 0$ and real $(\hat\lambda_q)_{1 \leq q \leq N}$. For notational convenience, we set $\hat\lambda_q = 0$ for $q > N$. We will restrict our attention to polynomials whose coefficients don't grow too quickly, in the sense of the following definition.
\begin{definition}\label{def-SubGaussPoly}
For a given $\beta > 0$, we say that a real, even, normalized polynomial $\lambda: \mbb{R} \to \mbb{R}$ is $(\Gamma,\eta,\theta)$-sub-Gaussian if
\[
|\hat\lambda_q| \leq \Gamma \cdot \exp\left[\left(\eta + \frac{\theta}{\beta}\right) q^2\right], \quad q \geq 1.
\]
\end{definition}
Throughout the paper, all trigonometric polynomials will be assumed to be real, even, normalized, and $(\Gamma, \eta, \theta)$-sub-Gaussian for some $\Gamma > 0$, $\eta \in \mbb{R}$, and $0 \leq \theta < 1/16$. Let $\Lambda$ be a square domain, $\beta > 0$, $h: \partial \Lambda \to \mbb R$, and $\lambda_\Lambda := (\lambda_j)_{j \in \Lambda^o}$ be a collection of trigonometric polynomials. We then define a (not necessarily positive) measure $\mu_{\beta,\Lambda,\lambda_\Lambda,h}$ on functions $\phi: \Lambda \to R$ by
\[
\mu_{\beta,\Lambda,\lambda_\Lambda,h}(A) = \frac{1}{Z_{\beta,\Lambda,\lambda_\Lambda,h}} \E^{\GFF-\Sym}_{\beta,\Lambda,h}\left[\one_{A}(\phi) \prod_{j \in \Lambda^o} \lambda_j(\phi_j)\right],
\]
where 
\[
Z_{\beta,\Lambda,\lambda_\Lambda,h} =\E^{\GFF-\Sym}_{\beta,\Lambda,h}\left[\prod_{j \in \Lambda^o} \lambda_j(\phi_j)\right].
\]
It follows from Theorem~\ref{thm-renormalization step} below that $Z_{\beta,\Lambda,\lambda_\Lambda,h} > 0$ so that $\mu_{\beta,\Lambda,\lambda_\Lambda,h}$ is well defined. We denote by $\E_{\beta,\Lambda,\lambda_\Lambda,h}$ the integration against $\mu_{\beta,\Lambda,\lambda_\Lambda,h}$ operation. 

The following theorem is an analog of \cite[Theorem 1.5]{KP17} for the field pinned at the boundary of a box (rather than at a single vertex as in the periodic or free boundary case).
\begin{theorem}\label{thm-MGF lower bound}
For any $\Gamma > 0$, $\eta \in \mbb{R}$, $0 \leq \theta < 1/16$, and $\epsilon > 0$, there exists $\beta_0 > 0$ such that the following holds. Let $\Lambda$ be a square domain, $h \in \Harm(\Lambda)$, $0 < \beta < \beta_0$, and $\lambda_\Lambda$ be a collection of $(\Gamma,\eta,\theta)$-sub-Gaussian polynomials. Let $\Lambda' \subset \Lambda$ be a square sub-domain such that $\dist(\Lambda',\partial\Lambda) \geq L/8$. If $f \in \Harm(\Lambda')^\perp$ and $\sigma = -(\Delta_\Lambda)^{-1}f$, then
\[
\E_{\beta,\Lambda,\lambda_\Lambda, h}[e^{\langle \phi,f\rangle}] \geq \exp\left[\frac{1}{2(1+\epsilon)\beta}\langle  \sigma, f\rangle\right].
\]
\end{theorem}
Assuming the theorem holds, the proof of Proposition~\ref{prop-IV MGF lower bound} is identical to the proof of \cite[Theorem 1.1]{KP17} given in \cite[Section 5.1]{KP17}. We include the details here for completeness. We begin by introducing some notation. Let $\mbb R^\Lambda$ (resp. $\mbb Z^\Lambda$) be the set of real-valued (resp. integer-valued) functions on $\Lambda$. For $h: \partial \Lambda \to \mbb R$ we let $\mbb R^\Lambda_{h}$ and $\mbb Z^{\Lambda}_{h}$ be the following sets
\begin{align*}
\mbb R^\Lambda_{h} &:= \{g \in \mbb R^{\Lambda} \,:\, g_j = h_j,\,\, j \in \partial \Lambda\},\\
\mbb Z^\Lambda_{h} &:= \{g \in \mbb R^{\Lambda}_{h} \,:\, g_j \in \mbb Z,\,\, j \in \Lambda^o\}.
\end{align*}
For $g \in \mbb Z^\Lambda_{h}$, we let $\Omega_{g} \subset \mbb R^\Lambda_{2\pi h}$ be the set of functions $\phi$ satisfying
\begin{align*}
\phi_j - 2\pi g_j &\in [-\pi,\pi),\quad j \in \Lambda^o.
\end{align*}
Finally, we let $(F_N)_{N \geq 1}$ be the Fej\'{e}r kernel
\[
F_N(x) := 1 + \sum_{q = 1}^{N-1}2 \left(1 - \frac{q}{N}\right) \cos(qx).
\]
Note that $F_N$ is a positive summability kernel (see \cite[Chapter 1]{Katznelson04}) so the following holds. For $g \in \mbb Z^\Lambda_{h}$ and $\Psi: \mbb R^\Lambda_{h} \to \mbb R$ a continuous function,
\begin{equation}\label{eq-summability kernel property}
\lim_{N \to \infty} (2\pi)^{-|\Lambda^o|}\int_{\Omega_g} \Psi(\phi) \prod_{j \in \Lambda^o} F_N(\phi_j) d\phi_j = \Psi(2\pi g).
\end{equation}
With these notations, we can prove the following.
\begin{lemma}\label{lm-weak convergence}
Let $\beta > 0$, $\Lambda$ be a square domain, and $h : \partial \Lambda \to \mbb Z$. Let $\lambda_{\Lambda,N}$ be the collection of polynomials such that $\lambda_j = F_N$ for all $j \in \Lambda^o$. For any $f \in \Harm(\Lambda)^\perp$,
\[
\lim _{N \to \infty} \E_{\beta/(2\pi)^2,\Lambda,\lambda_{N},2\pi h}\left[e^{\langle \phi,\frac{1}{2\pi}f\rangle}\right] = \E^{\IV-\Sym}_{\beta,\Lambda,h}\left[e^{\langle m, f\rangle}\right].
\]
\end{lemma}
Combined with Theorem~\ref{thm-MGF lower bound} and the fact that $F_N$ is $(1,0,0)$-sub-Gaussian for each $N \geq 1$, this lemma immediately implies Proposition~\ref{prop-IV MGF lower bound}.
\begin{proof}[Proof of Lemma~\ref{lm-weak convergence}]
Let $E^\pm$ be the following functions
\begin{align*}
E^{\pm}(f,N) = \E^{\GFF}_{\beta/(2\pi)^2,\Lambda, \pm 2\pi h}\left[e^{\langle \phi,\frac{1}{2\pi}f\rangle} \prod_{j \in \Lambda^o} F_N(\phi_j)\right].
\end{align*}
Then
\begin{equation}\label{eq-tilted MGF}
\E_{\beta/(2\pi)^2,\Lambda,\lambda_{N},2\pi h}\left[e^{\langle \phi,\frac{1}{2\pi}f\rangle}\right] = \frac{E^+(f,N) + E^-(f,N)}{E^+(\0,N) + E^-(\0,N)}.
\end{equation}
Recall that for $g \in \mbb R^\Lambda$
\[
\sum_{j \sim l} (g_j - g_l)^2 = \langle g, -\Delta_{\Lambda} g\rangle.
\]
Let $Z_{\beta, \Lambda,h}$ be as in \eqref{eq-GFF density} and note that $Z_{\beta,\Lambda, h} = Z_{\Lambda,\beta,-h}$. Therefore we have,
\begin{align*}
E^{\pm}(f,N) &= Z^{-1}_{\beta,\Lambda,h}\int_{\phi \in \mbb R^{\Lambda}_{\pm 2\pi h}} e^{\frac{1}{2\pi}\langle \phi,f\rangle - \frac{\beta}{2 (2\pi)^2}\langle \phi, - \Delta_\Lambda \phi\rangle} \prod_{j \in \Lambda^o} F_N(\phi_j)d\phi_j,\\
&= Z^{-1}_{\beta,\Lambda,h}\sum_{g \in \mbb Z^\Lambda_{\pm h}} \int_{\Omega_g} e^{\frac{1}{2\pi}\langle \phi,f\rangle - \frac{\beta}{2 (2\pi)^2}\langle \phi, - \Delta_\Lambda \phi\rangle} \prod_{j \in \Lambda^o} F_N(\phi_j)d\phi_j.
\end{align*}
Applying \eqref{eq-summability kernel property} with
\[
\Psi(\phi) = \exp\left(\frac{1}{2\pi}\langle \phi,f\rangle - \frac{\beta}{2(2\pi)^2}\langle \phi,-\Delta_\Lambda \phi\rangle\right),
\]
we obtain
\begin{align*}
\lim_{N \to \infty} E^{\pm}(f,N)
	&= \frac{(2\pi)^{|\Lambda^o|}}{Z_{\beta,\Lambda, h}}\sum_{g \in \mbb Z^\Lambda_{\pm h}} e^{\langle g, f\rangle - \frac{\beta}{2}\langle g, -\Delta_{\Lambda} g\rangle}\,.
\end{align*}
Plugging this into \eqref{eq-tilted MGF} we see
\[
\lim _{N \to \infty} \E_{\beta/(2\pi)^2,\Lambda,\lambda_{N},2\pi h}\left[e^{\langle \phi,\frac{1}{2\pi}f\rangle}\right]
	= \frac{1}{2 Z_{\beta,\Lambda,h}^{\IV}} \cdot \sum_{g \in \mbb{Z}^\Lambda_{\pm h'}} e^{\langle g, f\rangle - \frac{\beta}{2}\langle g, -\Delta_{\Lambda} g\rangle},
\]
where
\[
Z^{\IV}_{\beta,\Lambda,h} = Z^{\IV}_{\beta,\Lambda,-h} = \sum_{g \in \mbb Z^\Lambda_{h}} e^{-\frac{\beta}{2}\langle g, -\Delta_{\Lambda} g\rangle}.
\]
This concludes the proof.
\end{proof}
The rest of this section is therefore devoted to proving Theorem~\ref{thm-MGF lower bound}. By \eqref{eq-GFF MGF}
\begin{align*}
\E_{\beta,\Lambda,\lambda_\Lambda, h}[e^{\langle \phi, f \rangle}] &=
	Z_{\beta,\Lambda,\lambda_\Lambda, h}^{-1}\cdot \E^{\GFF-\Sym}_{\beta,\Lambda, h}\left[e^{\langle \phi, f \rangle} \prod_{j \in \Lambda^o} \lambda_j(\phi_j)\right]\\
	&= \exp\left(\frac{1}{2\beta}\langle \sigma, f\rangle\right)\cdot Z_{\beta,\Lambda,\lambda_\Lambda, h}^{-1}\cdot\E^{\GFF-\Sym}_{\beta,\Lambda, h} \left[\prod_{j \in \Lambda^o} \lambda_j\left(\phi_j+\frac{\sigma_j}{\beta}\right)\right],\\
	&= \exp\left(\frac{1}{2\beta}\langle \sigma, f\rangle\right)\cdot\frac{Z_{\beta,\Lambda,\lambda_\Lambda,h}(\beta^{-1}\sigma)}{Z_{\beta,\Lambda,\lambda_\Lambda,h}(\0)}
\end{align*}
where
\[
Z_{\beta,\Lambda,\lambda_\Lambda,h}(g) := \E^{\GFF-\Sym}_{\beta,\Lambda, h}\left[\prod_{j \in \Lambda^o} \lambda_j(\phi_j + g_j)\right].
\]
To conclude the proof, we need to show that for every $\epsilon > 0$ there exists $\beta_0 > 0$ such that the following holds for $\beta < \beta_0$
\[
\frac{Z_{\beta,\Lambda,\lambda_\Lambda,h}(\beta^{-1}\sigma)}{Z_{\beta,\Lambda,\lambda_\Lambda,h}(\0)} \geq \exp\left(- \frac{\epsilon}{2(1+\epsilon)\beta} \langle \sigma, f\rangle\right)\,.
\]

\subsection{Renormalization step}
The main step in the proof of Theorem~\ref{thm-MGF lower bound} consists of expressing the integral against $\mu_{\beta,\Lambda,\lambda_\Lambda,h}$ as a convex combination of integrals against positive measures. This is analogous to the proof of \cite[Theorem 1.5]{KP17} using \cite[Theorem 1.6]{KP17}. We begin with some definitions. The support of a function $g: \Lambda \to \mbb R$ is
\[
\supp(g) := \{j \in \Lambda \,:\, g_j \neq 0\}.
\]
A \emph{(charge) density} is a function $\rho: \Lambda \to \mbb Z$ such that $\supp(\rho) \neq \emptyset$ and $\supp(\rho) \subset \Lambda^o$. An \emph{ensemble} is a finite (possibly empty) collection of charge densities whose supports are mutually disjoint. The \emph{charge} $Q(\rho)$ of a density $\rho$ is defined by
\[
Q(\rho) := \sum_{j \in \Lambda^o} \rho_j.
\]
A density $\rho$ is called \emph{neutral} if $Q(\rho) = 0$; otherwise it is said to be \emph{charged} or \emph{non-neutral}. The diameter of a charge density is
\[
d(\rho) := \max \{\dist(j,l) \,:\, j,l \in \supp(\rho)\},
\]
where as usual $\dist$ denotes the graph distance. The following modified diameter will also be used in the proof
\[
d_{\Lambda}(\rho) = \begin{cases}
\max(d(\rho), \dist(\rho,\partial \Lambda)), &Q(\rho) \neq 0,\\
d(\rho) &Q(\rho) = 0.
\end{cases}
\]
Above and throughout the rest of the paper, $\dist(\rho, A) = \dist(\supp(\rho),A)$. Note that $d_\Lambda(\rho) \geq 1$ for any $\rho$. For each density $\rho$, if $d_{\Lambda}(\rho) = d(\rho)$ let $j \in \supp \rho$ be such that there exists $l \in \supp \rho$ with $d(\rho) = \dist(j,l)$. If $d_{\Lambda}(\rho) \neq d(\rho)$ let $j \in \supp \rho$ be such that $\dist(j,\partial \Lambda) = d_\Lambda(\rho)$. In both cases, $j$ is chosen in some fixed arbitrary way if there is more than one possible choice. We define
\[
D(\rho) = D_j(\rho) := \{ l \in \Lambda \,:\, \dist(j,l) < 2d_\Lambda(\rho)\},
\]
and say that $j$ is the \emph{center} of $D(\rho)$. Note that $\supp \rho \subset D(\rho)$. Finally, we denote
\[
||\rho||_2 := \sqrt{\sum_{j \in \Lambda^o} \rho_j^2}.
\]
With this notation in place, we can state the renormalization theorem. It is an analog of \cite[Theorem 1.6]{KP17} for the field with zero-boundary condition, and was stated with a detailed outline of the proof in the original paper of Fr\"{o}hlich and Spencer \cite[Appendix D]{FrohlichSpencer81}. For the reader's convenience, we provide the details of the proof (following the notation and presentation in \cite{KP17}) in the appendix.
\begin{theorem}\label{thm-renormalization step}
Let $\Gamma > 0$, $\eta \in \mbb{R}$, $0 \leq \theta < 1/16$. There exist constants $\beta_1, c_2 > 0$ such that the following holds. Let $\Lambda$ be a square domain of side-length $L$, $0 < \beta < \beta_1$, and $\lambda_\Lambda$ be a collection of $(\Gamma,\eta,\theta)$-sub-Gaussian polynomials. Then there exist:
	\begin{itemize}
	\item a finite collection of ensembles $\mc F$
	\item positive coefficients $(c_{\mc N})_{\mc N \in \mc F}$ summing to 1
	\item real coefficients $(z(\beta,\rho,\mc N))_{\rho \in \mc N, \mc N \in \mc F}$
	\item functions $a_\rho: \Lambda \to \mbb R$ for each $\rho \in \mc N$, $\mc N \in \mc F$
	\end{itemize}
such that for $g: \Lambda \to \mbb{R}$,
\[
\E^{\GFF}_{\beta,\Lambda,\0}\left[\prod_{j \in \Lambda} \lambda_j(\phi_j + g_j)\right] = \sum_{\mc N \in \mc F} c_{\mc N}\cdot\E^{\GFF}_{\beta,\Lambda,\0}\left[\prod_{\rho \in \mc N}[1+z(\beta,\rho,\mc N)\cos(\langle \phi, \rho + \Delta_\Lambda a_\rho\rangle + \langle g,\rho\rangle)]\right],
\]
and the following properties are satisfied for every $\mc N \in \mc F$:
\begin{enumerate}
\item For every $\rho \in \mc N$
\[
|z(\beta,\rho,\mc N)| \leq \exp\left[-\frac{c_2}{\beta}\left(||\rho||_2^2 + \log_2(d_{\Lambda}(\rho)+1)\right)\right].
\]
\item For distinct $\rho_1,\rho_2 \in \mc N$, if $d_\Lambda(\rho_1),d_\Lambda(\rho_2) \in [2^k -1, 2^{k+1}-2]$, $k \geq 1$, then $D(\rho_1) \cap D(\rho_2) = \emptyset$.
\item For $\Lambda' \subset \Lambda$ a sub-domain such that $\dist(\Lambda',\partial\Lambda) \geq L/8$, there exists at most one $\rho_c \in \mc N$ such that $\supp(\rho_c)\cap \Lambda' \neq \emptyset$ and $Q(\rho_c) \neq 0$.
\end{enumerate}
\end{theorem}
Before proceeding, we note a few immediate consequences of the theorem. First, for any boundary condition $h$ we have $\langle \tilde h , \Delta_\Lambda a_\rho\rangle = 0$, so the theorem is valid for the GFF with non-zero boundary condition (with the same family $\mc F$, coefficients $c_{\mc N}$, and functions $a_\rho$). It follows that it is valid for the symmetrized GFF as well.

\subsection{Finishing the proof of Theorem~\ref{thm-MGF lower bound}}
In this section we show how to obtain Theorem~\ref{thm-MGF lower bound} from Theorem~\ref{thm-renormalization step}. We have for $0 < \beta < \beta_1$,
\begin{align*}
Z_{\beta,\Lambda,\lambda_\Lambda,h}(g) = \sum_{\mc N \in \mc F} c_{\mc N}\cdot\E^{\GFF-\Sym}_{\beta,\Lambda,h}\left[\prod_{\rho \in \mc N}[1+z(\beta,\rho,\mc N)\cos(\langle \phi, \rho + \Delta_\Lambda a_\rho \rangle + \langle g,\rho\rangle)]\right]\,.
\end{align*}
Thus, we see that it suffices to prove that there exists $0 < \beta_2 \leq \beta_1$ such that  for every $\mc N \in \mc F$ and $\beta < \beta_2$
\[
\frac{Z_{\mc N}(\beta^{-1}\sigma)}{Z_{\mc N}(\0)}\geq \exp\left(- \frac{\epsilon}{2(1+\epsilon)\beta} \langle \sigma, f\rangle\right),
\]
where
\[
Z_{\mc N}(g) := \E^{\GFF-\Sym}_{\beta,\Lambda, h}\left[\prod_{\rho \in \mc N} [1 + z(\beta,\rho,\mc N)\cos(\langle \phi, \rho + \Delta_\Lambda a_\rho\rangle + \langle g, \rho\rangle)]\right].
\]
To do this we need the following claims
\begin{claim}[{\cite[Claim 3.1]{KP17}}]\label{claim-cos Taylor}
Let $x, y \in \mbb R$ and $0 < |z| < 1/8$. There exists an absolute constant $D_2 > 0$ such that
\[
1 + z\cos(x+y) \geq \exp\left(-\frac{z \sin x \sin y}{1 + z \cos x} - D_2|z| y^2\right)(1+z\cos x).
\]
\end{claim}
For the second claim, we let $\mc N_0 \subset \mc N$ be the set of neutral densities in $\mc N$ and $\mc N_c := \mc N \setminus \mc N_0$ be the set of charged densities in $\mc N$. We have
\begin{claim}[{\cite[Claim 3.1]{KP17}}]\label{claim-coefficient times inner product}
Let $D > 0$. There exists $0 \leq \beta_3 \leq \beta_1$ such that for $\beta < \beta_3$ and $\mc N \in \mc F$
\[
\sum_{\rho \in \mc N_0} |z(\beta,\rho,\mc N)|\cdot \langle \sigma, \rho\rangle^2 \leq \frac{\beta}{D} \sum_{j \sim l} (\sigma_j - \sigma_l)^2.
\]
\end{claim}
Additionally, we claim
\begin{claim}\label{claim-controlling non-neutral density}
Let $D > 0$. There exists $ 0 < \beta_4 \leq \beta_1$ such that for $\beta < \beta_4$ and $\mc N \in \mc F$
\[
\sum_{\rho \in \mc N_c} |z(\beta,\rho,\mc N)|\cdot \langle \sigma, \rho\rangle^2 \leq \frac{\beta}{D} \sum_{j \sim l} (\sigma_j - \sigma_l)^2.
\]
\end{claim}
The proof of this claim is given at the end of the section. First, we show how to conclude the proof of the theorem.

To simplify notation, we let $\bar\rho := \rho + \Delta_\Lambda a_\rho$. We let $\beta_5$ be small enough that $|z(\beta,\rho, \mc N)| < 1/8$ for all $\rho \in \mc N$, $\mc N \in \mc F$, and $\beta < \beta_5$. Then by Claim~\ref{claim-cos Taylor} we have for $\beta < \beta_5$
\[
\frac{Z_{\mc N}(\beta^{-1}\sigma)}{Z_{\mc N}(\0)} \geq \exp\left[-\frac{D_2}{\beta^2} \sum_{\rho \in \mc N}|z(\beta,\rho,\mc N)| \langle \sigma, \rho \rangle^2\right] \cdot \E_{\mc N}\left[e^{S(\mc N, \phi)}\right],
\]
where $S(\mc N, \phi)$ is the function
\[
S(\mc N, \phi) := - \sum_{\rho \in \mc N} \frac{z(\beta,\rho,\mc N)\sin(\langle \phi, \bar \rho\rangle) \sin \langle \beta^{-1}\sigma, \rho \rangle}{1 + z(\beta,\rho,\mc N)\cos(\langle \phi, \bar \rho\rangle)},
\]
and $\mbb P_{\mc N}$ is a probability measure given by
\[
\mbb P_{\mc N}(A) = Z_{\mc N}(\0)^{-1}\cdot \E^{\GFF-\Sym}_{\beta,\Lambda, h}\left[\one_{A}(\phi)\prod_{\rho \in \mc N} [1 + z(\beta,\rho,\mc N)\cos(\langle \phi, \bar \rho\rangle)]\right].
\]
We note that $S(\mc N, \phi) = - S(\mc N, -\phi)$ and that $\mbb P_{\mc N}$ invariant under the mapping $\phi \to -\phi$. Therefore, by Jensen's inequality $\E_{\mc N}[e^{S(\mc N, \phi)}] \geq 1$. Thus, applying Claim~\ref{claim-coefficient times inner product} and Claim~\ref{claim-controlling non-neutral density} with $D := \frac{4(1+\epsilon)}{\epsilon} D_2$ we conclude that for $\beta < \beta_2 := \beta_3 \wedge \beta_4 \wedge \beta_5$
\begin{align*}
\frac{Z_{\mc N}(\beta^{-1}\sigma)}{Z_{\mc N}(\0)}
	&\geq \exp\left[-\frac{D_2}{\beta^2} \sum_{\rho \in \mc N}|z(\beta,\rho,\mc N)| \langle \sigma, \rho \rangle^2\right] \\
	&\geq \exp\left(-\frac{\epsilon}{2(1+\epsilon)\beta}\sum_{j \sim l}(\sigma_j - \sigma_l)^2\right).
\end{align*}
Recalling that $\sum_{j \sim l} (\sigma_j - \sigma_l)^2 = \langle \sigma, -\Delta_\Lambda \sigma\rangle = \langle \sigma, f \rangle$ concludes the proof of Theorem~\ref{thm-MGF lower bound}.

\begin{proof}[Proof of Claim~\ref{claim-controlling non-neutral density}:]
By property 3 in Theorem~\ref{thm-renormalization step}, there is at most one density $\rho_c \in \mc N_c$ such that $\langle \rho_c, \sigma \rangle \neq 0$. We let $k \in \Lambda$ be a vertex such that $\dist(k, \Lambda') = 1$ and $\rho_c'$ be a charge density such that $\rho_{c,j}' = \rho_{c,j}$ for $j \in (\Lambda')^o$, $\rho'_{c,k} = -Q(\rho_c)$, and $\rho'_{c,j} = 0$ for $j \in \Lambda \setminus (\Lambda' \cup \{k\})$. Thus, $\langle \rho_c ,\sigma \rangle = \langle \rho_c',\sigma \rangle$ and $Q(\rho_c') = 0$. From this point the proof is the same as the proof of \cite[Claim 3.2]{KP17}. Since $\rho_c'$ takes integer values and is neutral, there exist integer values $(c_{\{j,l\}})_{j \sim l}$ such that
\[
\langle \rho_c', \sigma\rangle = \sum_{\substack{j,l \in D(\rho_c')\\ j \sim l}} c_{\{j,l\}}(\sigma_j - \sigma_l),
\]
and $|c_{j,l}| \leq \frac{1}{2}\sum_{j \in \Lambda} |\rho'_{c,j}| \leq \sum_{j \in \Lambda} |\rho_{c,j}| \leq ||\rho_c||_2^2$. This can be proved by induction on $||\rho_c'||_1 : = \sum_{j \in \Lambda} | \rho_{c,j}'|$ since a netural density with $||\rho'_c||_1 > 2$ can be decomposed into a sum of two netural densities $\rho_1, \rho_2$ with $||\rho_1||_1 , ||\rho_2||_1 < ||\rho_c'||_1$, and the case $||\rho'_c||_1 = 2$ is trivial. Applying the Cauchy-Schwarz inequality gives
\[
\langle \rho_c', \sigma\rangle^2 \leq 4 |D(\rho_c')|\, ||\rho_c||_2^4 \sum_{j \sim l} (\sigma_j - \sigma_l)^2.
\]
Note that as $\supp(\rho_c)\cap \Lambda' \neq \emptyset$ and $Q(\rho_c) \neq 0$, we have 
\[
d_{\Lambda}(\rho_c) \geq \frac{L}{16}.
\]
By property 1 in Theorem~\ref{thm-renormalization step} and the trivial bound $|D(\rho_c')| \leq L^2$ we conclude that there exists $0 < \beta_4 \leq \beta_1$ such that for $0 < \beta < \beta_4$
\[
|z(\beta,\rho_c,\mc N)| \langle \rho_c, \sigma\rangle^2  \leq \frac{\beta}{D} \sum_{j \sim l} (\sigma_j - \sigma_l)^2. \qedhere
\]
\end{proof}
\small

\section{Proof of Theorem~\ref{thm-max asymptotic no rate}}\label{sec-thm2 proof}
\subsection{Overview of the proof}
Let $\epsilon > 0$ be fixed, and $N > 0$ be the smallest integer such that $(\frac{3}{4})^N < \frac{\epsilon}{2}$. Let $D_1$ and $\beta_0$ be as in Proposition~\ref{prop-maximum epsilon proba}. We assume for the rest of the section that $0 < \beta < \beta_0$. Let $\delta = (3 D_1)^{-1}$ and for $k = 1,\dots N$ let 
\begin{equation}\label{eq-delta sequence}
\delta_k = 9 \left(\frac{\delta}{9}\right)^{2^{N-k}}.
\end{equation}
We define the sets $A_k$ and $\Lambda_k$ recursively as follows. Take $\Lambda_1 := \Lambda$ and for $k = 1,\dots, N$,
\[
A_k = \{j \in \Lambda_{k} \,:\, \dist(j,\partial \Lambda_k) \leq L^{\delta_k}\}, \quad \Lambda_{k+1} = \Lambda_k \setminus A_k.
\]
Note that $\Lambda_k$ is a sub-domain of $\Lambda$. Let $R_0$ be as in Proposition~\ref{prop-maximum epsilon proba} and $L_3$ be the smallest integer such that $L_3^{\delta_1} \geq R_0$. We assume from now on that $L \geq L_3$. We note that under this assumption, since $R_0 \geq 10$, $\delta < 1/10$, and $N \geq 3$, one can easily check that the side-length of $\Lambda_N$ is at least $L/2$. Next, we let we let $b_k$ be given by
\begin{align*}
b_1 &= 0, \quad b_k = b_{k-1} + 2\sqrt{\delta_{k-1}}.
\end{align*}
Note that for $1 \leq k \leq N-1$ we have $\delta_{k+1} = 3 \sqrt{\delta_k}$ and consequently
\[
\delta_{k+1} - b_{k+1} = \sqrt{\delta_k} - b_k > \delta_k - b_k.
\]
Therefore
\[
c_1 := \min_{1 \leq k \leq N} (\delta_k - b_k) = \delta_1.
\]
Finally, we introduce the events
\[
\mc E_{k,1} = \left\{\ \max_{j \in A_k} m_j \geq \frac{\delta_k-b_{k}}{\sqrt{\beta}} \log(L)\right\}, \quad \mc E_{k,2} = \left\{\min_{j \in \partial \Lambda_k} m_j \geq - \frac{b_{k}}{\sqrt{\beta}} \log(L)\right\}.
\]
Note that $\mbb P^{\IV}_{\beta,\Lambda,\0}(\mc E_{1,2}) = 1$. We claim that there exists $L_4 \geq L_3$ such that for $L \geq L_4$ the following holds for $k = 1,\dots, N$
\begin{equation}\label{eq-good event bootstrap}
\mbb P^{\IV}_{\beta,\Lambda,\0}\left(\mc E_{k,1} \mid \mc E_{k,2}\right) \geq \frac{1}{4}, \quad\text{and}\quad
\mbb P^{\IV}_{\beta,\Lambda,\0}\left(\mc E_{k+1,2}^c \mid \mc E_{k,2}\right) \leq L^{-1}.
\end{equation}
Assuming this holds, we conclude that for $L \geq L_4$ 
\[
\mbb P^{\IV}_{\beta,\Lambda,\0}\left( \max_{j \in \Lambda} m_j < \frac{c_1}{\sqrt{\beta}} \log(L) \right) \leq \left(\frac{3}{4}\right)^N + 4 L^{-1} <\epsilon,
\]
where we used the fact that $L \geq L_3$ implies that $L > \frac{8}{\epsilon}$. Thus, it remains to prove \eqref{eq-good event bootstrap}.

\subsection{Proof of \eqref{eq-good event bootstrap}}
To control the effect of conditioning on $\mc E_{k,2}$, we use the fact that the law of the IV-GFF is increasing as a function of the boundary condition. To state this fact, we introduce some notation.

We let $\mbb Z^\Lambda$ be the space of integer-valued functions on $\Lambda$ and similarly for $\mbb Z^{\partial \Lambda}$. For two functions $m_1, m_2 \in \mbb Z^{\Lambda}$, we say $m_2$ is larger than $m_1$, and write $m_1 \leq m_2$, if $m_1(j) \leq m_2(j)$ for all $j \in \Lambda$. We call a function $g: \mbb Z^\Lambda \to \mbb R$ increasing if it is increasing in each coordinate. That is, $g(m_1) \leq g(m_2)$ whenever $m_1 \leq m_2$. Finally, for two probability measures $\mu_1,\mu_2$ on $\mbb Z^\Lambda$, we say $\mu_2$ is stochastically larger than $\mu_1$, and write $\mu_1 \leq \mu_2$, if the following holds for all bounded, increasing functions $g: \mbb Z^\Lambda \to \mbb R$,
\[
\sum_{m \in \mbb Z^\Lambda} \mu_1(m) g(m) \leq \sum_{m \in \mbb Z^\Lambda} \mu_2(m) g(m).
\]
The following lemma shows that the law of the IV-GFF is increasing as a function of the boundary condition in this sense
\begin{lemma}\label{lm-bdry condition monotonicity}
Let $\Lambda$ be a square domain and $h_1,h_2 \in \mbb Z^{\partial \Lambda}$ be such that $h_1 \leq h_2$. Then $\mbb P^{\IV}_{\beta,\Lambda,h_1} \leq \mbb P^{\IV}_{\beta,\Lambda,h_2}$ for all $\beta > 0$.
\end{lemma}
\begin{proof}[Proof of Lemma~\ref{lm-bdry condition monotonicity}:]
We provide a sketch of the proof here. To simplify notation, we let $\mu_i = \mbb P^{\IV-\GFF}_{\beta,\Lambda,h_i}$ for $i = 1,2$. We consider $\mu_1$ and $\mu_2$ as probability measures on $Z^{\Lambda^o}$. For $m,m' \in \mbb Z^{\Lambda^o}$, we let $m \wedge m' \in \mbb Z^{\Lambda^o}$ be given by $(m \wedge m')_j = \min(m_j,m'_j)$, and similarly $(m \vee m')_j = \max(m_j,m'_j)$. It follows from a simple calculation that $\mu_1$ and $\mu_2$ satisfy the Holley criterion. That is, for $m,m' \in \mbb Z^{\Lambda^o}$, we have
\[
\mu_1(m \wedge m')\mu_2(m \vee m') \geq \mu_1(m) \mu_2(m'). 
\]
From this it follows easily that the following holds. For $a \in \mbb Z$ and $v \in \Lambda^o$, let $I_{a,v} \subset \mbb Z^{\Lambda^o}$ be the subset of functions $m \in \mbb Z^{\Lambda^o}$ such that $m_v \geq a$. For $m \in \mbb Z^{\Lambda^o}$, let $J_{v,m}$ be the subset of functions $m' \in \mbb Z^{\Lambda^o}$ such that $m'_j = m_j$ for $j \neq v$. We have that for any $m_1,m_2 \in \mbb Z^{\Lambda^o}$ such that $m_1 \leq m_2$,
\[
\mu_1(I_{a,v} \mid J_{v, m_1}) \leq \mu_2(I_{a,v} \mid J_{v, m_2}). 
\]
That is, $\mu_1(\cdot \mid J_{v,m_1}) \leq \mu_2(\cdot \mid J_{v,m_2})$. From this one can show, by considering two coupled Markov chains with invariant distribution $\mu_1$ and $\mu_2$, that $\mu_1 \leq \mu_2$. See, e.g. \cite[Theorem 4.8]{GHM99}, for further details of the proof. 
\end{proof}
By noting that the minimum and maximum functions are increasing, we obtain the following corollary.
\begin{corollary}\label{cor-monotonicity needed}
Let $\Lambda$ be a square domain, $\beta, x > 0$ and $A \subset \Lambda$. Then
\[
\mbb{P}^{\IV}_{\beta,\Lambda,h}\left(\max_{j \in A} m_j \geq x\right),
\]
and
\[
\mbb{P}^{\IV}_{\beta,\Lambda,h}\left(\min_{j \in A} m_j \geq -x\right),
\]
are increasing as functions of the boundary condition $h : \partial \Lambda \to \mbb Z$.
\end{corollary}
Provided with this fact, the proof proceeds as follows. Let $\mathbf 1_k : \partial \Lambda_k \to \mbb R$ be the function identically equal to 1 and $\nu = \beta^{-1/2}\log(L)$. We have by Lemma~\ref{lm-domain Markov property} and Corollary~\ref{cor-monotonicity needed} 
\[
\mbb P^{\IV}_{\beta,\Lambda,\0}\left(\mc E_{k,1} \mid \mc E_{k,2}\right) \geq
	\mbb P^{\IV}_{\beta,\Lambda_k,-b_k \nu\mathbf 1_k}(\mc E_{k,1})
\]
and
\[
\mbb P^{\IV}_{\beta,\Lambda,\0}\left(\mc E_{k+1,2}^c \mid \mc E_{k,2}\right) \leq
	\mbb P^{\IV}_{\beta,\Lambda_k,-b_k \nu\mathbf 1_k}(\mc E^c_{k+1,2})\,.
\]
Note that if $m$ has law $\mbb P^{\IV}_{\beta,\Lambda_k,-b_k \nu\mathbf 1_k}$ then $m'$ defined by $m'_j = m_j + b_k\nu$ has law $\mbb P^{\IV}_{\beta,\Lambda_k,\0}$. Thus, it suffices to show that
\begin{align}
\mbb P^{\IV}_{\beta,\Lambda_k,\0}\left(\max_{j \in A_k} m_j \geq \delta_k \nu\right) &\geq \frac{1}{4} \label{eq-maximum positive proba}\\
\mbb P^{\IV}_{\beta,\Lambda_k,\0}\left(\min_{j \in \partial \Lambda_{k+1}} m_j < - 2\sqrt{\delta_{k}} \nu\right) &\leq L^{-1} \label{eq-minimum negligible}.
\end{align}
We begin with \eqref{eq-maximum positive proba}. Let $Q_k$ be a collection of disjoint sub-domains of $\Lambda_k$ of side-length $R = \lceil L^{\delta_k}\rceil$ which are contained in $A_k$ satisfying $|Q_k| \geq L^{1-\delta_k}$. That such a collection exists follows from the fact that the side-length of $\Lambda_k$ is at least $L/2$ and $L^{\delta_k} \geq 10$. For each $\Pi \in Q_k$, let
\[
U_\Pi := \mbb P^{\IV}_{\beta,\Lambda_k,\0} \left(\max_{j \in \Pi} |m_j| \geq \delta_k \nu \mid m_j, \, j \in \partial \Pi\right).
\]
Recall that by assumption $\delta_k \leq \delta = (3D_1)^{-1}$, and that $D_1 \geq 100$. Therefore, if we let
\[
w_k = \sqrt{\frac{2D_1}{\delta_k}},
\]
we have
\[
1 - (1 + w_k)\delta_k \geq 1 - \sqrt{\frac{2}{3}} - \sqrt{\delta} > \frac{1}{10}, \quad \text{ and } \quad 1 - \frac{w_k^2}{D_1}\delta_k = -1.
\]
Noting that $|Q_k| \leq L$ and applying Proposition~\ref{prop-maximum epsilon proba} with $w = w_k$ to each element of $Q_k$ we see that the event 
\[
\mc E_{k,3} = \left\{\min_{\Pi \in Q_k} U_\Pi \geq R^{-w_k}\right\}
\]
satisfies
\[
\mbb P^{\IV}_{\beta,\Lambda_k,\0} (\mc E_{k,3}^c) \leq |Q_k| L^{-\frac{w_k^2}{D_1} \delta_k} \leq L^{-1},
\]
and
\[
\mbb P^{\IV}_{\beta,\Lambda_k,\0} \left(\max_{j \in A_k} |m_j| < \delta_k \nu \mid \mc E_{k,3}\right) \leq \left(1 - (L^{\delta_k} + 1)^{-w_k}\right)^{|Q_k|} \leq \exp\left(-L^{1/20}\right),
\]
where we used the fact that $(L^{\delta_k} + 1)^{w_k} \leq L^{\delta_k w_k + 1/20}$. Therefore, we have for $L \geq L_3$ (where $L_3$ is as above),
\[
\mbb P^{\IV}_{\beta,\Lambda_k,\0} \left(\max_{j \in A_k} |m_j|  \geq \delta_k \nu\right) \geq \frac{1}{2}.
\]
Thus, \eqref{eq-maximum positive proba} follows by symmetry. Next, we turn to \eqref{eq-minimum negligible}. We begin by noting that $|\partial \Lambda_{k+1}| \leq 4 L$ for all $k$, and that by \eqref{eq-Green asymptotic} there exists $L_4 \geq 0$ such that for $L \geq L_4$, 
\[
\sup_{j \in \partial \Lambda_{k+1}} G_{\Lambda_k}(j,j) \leq \frac{\delta_k}{2}\log(L).
\]
Therefore, we have by Proposition~\ref{prop-IV MGF upper bound} and a Chernoff bound that
\begin{align*}
\mbb P^{\IV}_{\beta,\Lambda'_k,\0} \left(\min_{j \in \partial \Lambda_k} m_j < -2 \sqrt{\delta_k} \nu\right) &\leq 4 L \sup_{j \in \partial \Lambda_k} \mbb P^{\IV}_{\beta,\Lambda'_k,\0} \left(m_j < -2 \sqrt{\delta_k} \nu\right)\\
&\leq 4L \min_{s > 0} \exp\left(\frac{s^2\delta_k}{4\beta} \log(L) - 2s\sqrt{\frac{\delta_k}{\beta}} \log(L)\right)\\
&\leq 4L^{-3} \leq L^{-1}.
\end{align*}

\subsection{Proof of Theorem~\ref{thm-max asymptotic no rate} for free and periodic boundary conditions}
We will outline the changes necessary to the argument given above to establish Theorem~\ref{thm-max asymptotic no rate} for the field with free or periodic boundary conditions. The main input we require are bounds on the Green's function in each case. This is given by the following proposition
\begin{proposition}
	There exist constants $c_4, c_5 > 0$ such that the following holds. Let $\Lambda$ be a square domain with free or periodic boundary. Let $A \subset \Lambda$ be non-empty. $S$ be a random walk on $\Lambda$, $\zeta = \min\{t \geq 0, \:\, S_t \in A\}$, and 
	\[
	G_{\Lambda\setminus A}(j,l) := \E\left[\int_0^{\zeta} \one_{\{l\}}(S_t)dt\right]
	\]
	be the Green's function on $\Lambda \setminus A$. Let $j \in \Lambda \setminus A$. Then
	\[
	c_4\log(\dist(v,A) +1) \leq G_{\Lambda \setminus A}(j,j) \leq c_5 \log(\dist(v,A) + 1).
	\]
\end{proposition}
These bounds are well-known and are proved by considering the extreme cases of $A = \{v\}$ (for the upper bound) and $A = \{v \in \Lambda \,:\, \dist(v,j) \geq k\}$ for some $k$ (for the lower bound).

Provided with these bounds, the proof proceeds as follows. First, we note that for any two vertices $v,v' \in \Lambda$, if $m$ is an IV-GFF pinned at $v$ (i.e. such that $m_v = 0$) then $m'$ defined by $m'_j = m_j - m_{v'}$ is an IV-GFF (at the same inverse temperature) pinned at $v'$. Further, we have by an application of Proposition~\ref{prop-IV MGF upper bound} and Markov's inequality that for any constant $c > 0$,
\[
\mbb{P}^{\IV}_{\beta,\Lambda,v} (m_{v'} \geq c \log(L)) \leq L^{-\frac{c^2\beta}{c_5}}.
\]
Therefore, we may assume without loss of generality that the field is pinned at the center of $\Lambda$. That is, at the vertex $v^* = (\lfloor \frac{L}{2}\rfloor, \lfloor \frac{L}{2}\rfloor)$.

Next, let $N$ be such that $(\frac{3}{4})^N < \frac{\epsilon}{2}$, and $\{\delta_k\}_{k=1}^N$ be the sequence satisfying $\delta_N = \frac{1}{2}$ and
\[
\delta_{k+1} = 3 D_1 \sqrt{c_5 \delta_{k}} + \delta_k, \quad k = 1,\dots, N-1.
\]
Let $\{b_k\}_{k = 1}^N$ be the sequence satisfying $b_1 = 0$ and
\[
b_{k+1} = b_{k} + 3\sqrt{c_5 \delta_{k}}.
\]
Note that with these choices $\delta_k/D_1 - b_k$ is increasing. Let $\Lambda_0 = \{v^*\}$ and for $k = 1,\dots,N$, let $\Lambda_k$ be a box of side-length $\lceil L^{\delta_k}\rceil$ centered at $v^*$, $A_k = \Lambda_k \setminus \Lambda_{k-1}$, and $B_k = \partial \Lambda_k$. Let $\gamma = (D_1)^{-1}$ and assume that $L^{\gamma\delta_1} \geq R_0$. Arguing as in Section~\ref{sec-thm1 sketch} one can show that
\[
\mbb{P}^{\IV}_{\beta,\Lambda,v^*}\left(\max_{j \in A_k} |m_j| \geq \frac{\gamma\delta_k}{\sqrt{\beta}}\log(L) \mid m_j = 0,\,\, \forall j \in \partial B_{k-1}\right) \geq \frac{1}{2}.
\]
Further, we can use the upper bound on the Green's function and a union bound (noting $|B_k| \leq 4 L^{\delta_k}$) to obtain
\[
\mbb{P}^{\IV}_{\beta,\Lambda,v^*}\left(\min_{j \in B_k} m_j < - \frac{3\sqrt{c_5 \delta_k}}{\sqrt{\beta}}\log(L) \mid m_j = 0,\,\, \forall j \in \partial B_{k-1}\right) \leq L^{-1}.
\]
Therefore, if we let
\[
\mc E_{k,1} = \left\{ \max_{j \in A_k} m_j \geq \frac{\gamma\delta_k -  b_k}{\sqrt{\beta}} \log(L)\right\}, \quad \mc E_{k,2} \left\{\min_{j \in B_k} \geq -\frac{b_k}{\sqrt{\beta}}\log(L)\right\}\,,
\]
we have as in the zero-boundary case
\[
\mbb{P}^{\IV}_{\beta,\Lambda,v^*}(\mc E_{k-1} \mid \mc E_{k-2}) \geq \frac{1}{4}, \quad \text{and} \quad \mbb P^{\IV}_{\beta,\Lambda,v^*}(\mc E_{k+1,2}^c \mid \mc E_{k,2}) \leq L^{-1}.
\]
This concludes the proof.

\appendix
\section{Appendix: Proof of Theorem~\ref{thm-renormalization step}}\label{sec-renormalization sketch}
In this section we provide the details involved in proving the renormalization result. As mentioned, this proof was outlined in \cite[Appendix D]{FrohlichSpencer81}. We fix $\beta$, $\Lambda$, $\Lambda'$, and $\lambda_\Lambda$ such that they satisfy the conditions of Theorem~\ref{thm-renormalization step}. We fix also
\[
\frac{3}{2} < \alpha < 2, \quad M = 2^{16}.
\]
\subsection{Square-Covering of Densities}
Let $k \geq 0$ be an integer. We call $s \subseteq \Lambda$ a $2^k \times 2^k$ square if $|s| = \min(|\Lambda|, 2^{2k})$ and there exists a vertex $(a,b) \in \Lambda$ such that
\[
s = \{(c,d) \in \Lambda \,:\, (c-a,d-b) \in \{0,1,\dots,2^k-1\}^2\}.
\]
For an integer $k \geq 0$ we let $\mathcal S_k(\rho)$ be a minimal collection of $2^k \times 2^k$ squares covering the support of $\rho$. Here minimal means of smallest cardinality. The choice of $\mathcal S_k(\rho)$ when more than one minimal cover exists is made in the same way as in \cite[Section 4.3.4]{KP17}. We then define $n(\rho) := \lceil \log_2(M \cdot d_\Lambda(\rho)^\alpha)\rceil$ and
\[
A(\rho) := \sum_{k = 0}^{n(\rho)}|\mathcal S_k(\rho)|.
\]
Next, we define $\mathcal S_k^{\sep}(\rho)$, $k\geq 1$ as follows. If $|\mathcal S_k(\rho)| > 1$, we let
\[
\mathcal S_k^{\sep}(\rho) := \{s \in \mathcal S_k(\rho)\,:\, \dist(s,s') \geq 2M2^{\alpha(k+1)}\,\, \forall s' \in \mathcal S_k(\rho) \setminus \{s\}\}.
\]
If $|\mathcal S_k(\rho)| = 1$ and $Q(\rho) = 0$, then we let $S_k^{\sep}(\rho) = \emptyset$. If $|\mathcal S(\rho)| = 1$ and $Q(\rho) \neq 0$, we let $\mathcal S_k^{\sep}(\rho) = \mathcal S_k(\rho)$ if $\dist(\rho, \partial \Lambda) \geq 2^{k+1}$ and $\mathcal S_k^{\sep}(\rho) = \emptyset$ otherwise. The following proposition will allow us to control the size of $A(\rho)$. The statement is slightly stronger than \cite[Proposition 2.1]{KP17} but it follows easily from that result.
\begin{proposition}[{\cite[Proposition 2.1]{KP17}}]
There exists a positive absolute constant $D_3$ such that for any density $\rho$,
\[
\log_2(d_{\Lambda}(\rho)+1) \leq A(\rho) \leq D_3\cdot \left( |\mathcal S_0(\rho)| + \sum_{k = 1}^{n(\rho)} |\mathcal S_k^{\sep}(\rho)| \right).
\]
\end{proposition}
\begin{proof}
Note that the lower bound is immediate since $A$ is defined as the sum of at least $\log_2(d_{\Lambda}(\rho)+1)$ terms, each bounded below by 1. We briefly explain how to obtain the upper bound from \cite[Proposition 2.1]{KP17}. The case $d(\rho) = 0$ (that is, $\rho$ is supported at a single vertex) is again immediate as $A(\rho) = n(\rho)$ and the right-hand side of the inequality is equal to $D_3 \cdot \max(\lfloor \log_2(d_{\Lambda}(\rho)) \rfloor,1)$. Therefore we assume $d(\rho) > 0$. Define $n'(\rho) := \lceil \log_2(M \cdot d(\rho)^\alpha) \rceil$ and
\[
A'(\rho) := \sum_{k = 0}^{n'(\rho)}|\mathcal S_k(\rho)|
\]
By \cite[Proposition 2.1]{KP17}, there exists an absolute constant $D > 0$ such that
\[
A'(\rho) \leq D\cdot \left( |\mathcal S_0(\rho)| + \sum_{k = 1}^{n'(\rho)} |\mathcal S_k^{\sep}(\rho)| \right)\,.
\]
Note $A(\rho) - A'(\rho) = n(\rho) - n'(\rho) \leq \log_2(d_{\Lambda}(\rho)/d(\rho)) + 1$. Let $m(\rho) = \lfloor \log_2(d_{\Lambda}(\rho)) -1 \rfloor$ and $m'(\rho) = \lceil \log_2(d(\rho)) + 1 \rceil$ and note that for $m'(\rho) \leq k \leq m(\rho)$, $|\mc S_k(\rho)| = |\mc S_k^{\sep}(\rho)| = 1$. Since $m(\rho) - m'(\rho) \geq \log_2(d_{\Lambda}(\rho)/d(\rho)) - 4$, we conclude that there exists a constant $D' > 0$ such that
\[
A(\rho) - A'(\rho) \leq D' \left(1 + \sum_{k = m'(\rho)}^{m(\rho)} |\mathcal S_k^{\sep}(\rho)| \right)\,.
\]
Combining these two bounds yields the desired result.
\end{proof}
\subsection{Expansion as a convex combination}
The following result is analogous to \cite[Theorem 2.2]{KP17}. We specify the necessary changes to the proof of \cite[Theorem 2.2]{KP17} needed to prove this version later in the section. Below and throughout the rest of the paper we write $\rho_1 \subset \rho$ if $\supp(\rho_1) \subset \supp(\rho)$, and $\rho_{1,j} = \rho_j$ for all $j \in \supp(\rho_1)$. In this case we say $\rho_1$ is a \emph{constituent} of $\rho$. 
\begin{theorem}\label{thm-1st convex expansion}
There exists a positive absolute constant $D_4$, a finite collection of ensembles $\mc F$, positive coefficients $(c_{\mc N})_{\mc N \in \mc F}$ summing to 1 and real $(K(\rho))_{\rho \in \mc N, \mc N \in \mc F}$, such that for every $\psi: \Lambda \to \mbb R$,
\[
\prod_{j \in \Lambda^o} \lambda_j(\psi_j) = \sum_{\mc N \in \mc F} c_{\mc N} \prod_{\rho \in \mc N}[1 + K(\rho) \cos(\langle \psi,\rho\rangle)].
\]
Additionally, the following properties are satisfied for each $\mc N \in \mc F$:
\begin{enumerate}[(a)]
	\item If $\rho, \rho' \in \mc N$ are distinct then $\dist(\rho,\rho') \geq M [\min(d_{\Lambda}(\rho), d_{\Lambda}(\rho'))]^\alpha$.
	\item If $\rho_1 \subset \rho \in \mc N$, $\rho_1 \neq \rho$, satisfies $\dist(\rho_1, \rho-\rho_1) \geq 2M d(\rho_1)^\alpha$, then $Q(\rho_1) \neq 0$ and
	\[
	2M \dist(\rho_1, \partial \Lambda)^\alpha > \dist(\rho_1, \rho - \rho_1).
	\]
	\item The coefficients $K(\rho)$ satisfy
	\[
	|K(\rho)| \leq e^{D_4 A(\rho)} \prod_{j \in \supp(\rho)} e^{\rho_j^2} |\hat\lambda_{j, |\rho_j|}|,
	\]
	where $\hat\lambda_{j,q}$ is the $q$th coefficient of the polynomial $\lambda_j$.
\end{enumerate}
\end{theorem}
Recall that $\dist(\Lambda', \partial \Lambda) \geq L/8$, so that any non-neutral density $\rho$ such that $\supp(\rho) \cap \Lambda' \neq \emptyset$ satisfies $d_{\Lambda}(\rho) \geq L/16$. Therefore, if $\mc N$ is an ensemble satisfying property (a) of Theorem~\ref{thm-1st convex expansion}, it contains at most one such density $\rho$, since the distance between any two such densities would exceed $2L = d(\Lambda)$. 

\subsection{Bounding the coefficients}
In this section we discuss how to modify the terms obtained by applying Theorem~\ref{thm-1st convex expansion} with $\psi = \phi + g$ to complete the proof of Theorem~\ref{thm-renormalization step}. In particular, we want to replace the coefficients $K(\rho)$ in terms of the form $\prod_{\rho \in \mc N} [1 + K(\rho)\cos(\langle \phi,\rho\rangle + \langle g,\rho\rangle)]$ with coefficients $z(\beta,\rho,\mc N)$ satisfying property 1 in Theorem~\ref{thm-renormalization step}. This will ensure in particular that the measure associated with such terms is positive.

\begin{theorem}\label{thm-coefficient modification}
There exists an absoute constant $D_5 > 0$ such that the following holds. Let $\mc N$ be an ensemble satisfying conditions $(a)$ and $(b)$ of Theorem~\ref{thm-1st convex expansion} and $(K(\rho))_{\rho \in \mc N}$ be real. Then there exist real $z(\beta,\rho,\mc N)$ and functions $a_{\rho,\mc N} : \Lambda \to \mbb R$ such that for every $g: \Lambda \to \mbb R$
\begin{align*}
&\int \prod_{\rho \in \mc N} [1 + K(\rho)\cos(\langle \phi,\rho\rangle + \langle g,\rho\rangle)] d\mbb P^{\GFF}_{\beta,\Lambda,\0}(\phi) \\
&= \int \prod_{\rho \in \mc N} [1 + z(\beta,\rho,\mc N)\cos(\langle \phi,\rho + \beta \Delta_{\Lambda} a_{\rho,\mc N}\rangle + \langle g,\rho\rangle)] d\mbb P^{\GFF}_{\beta,\Lambda,\0}(\phi).
\end{align*}
Additionally
\[
|z(\beta,\rho,\mc N)| \leq |K(\rho)|\exp\left[-\frac{1}{\beta}\left( \frac{1}{16}||\rho||_2^2 + D_5 \cdot \sum_{k = 1}^{n(\rho)} |\mc S^{\sep}_k(\rho)|\right)\right].
\]
If $(K(\rho))_{\rho \in \mc N}$ satisfy condition $(c)$ of Theorem~\ref{thm-1st convex expansion}, then there exist constants $\beta_0, c_2 > 0$ depending on $(\Gamma, \eta, \theta)$ only such that
\[
|z(\beta,\rho,\mc N)| \leq \exp\left[ - \frac{c_2}{\beta}\left(||\rho||_2^2 + \log_2(d_\Lambda(\rho)+1)\right)\right], \quad 0 < \beta < \beta_0\,.
\]
\end{theorem}
Theorem~\ref{thm-renormalization step} is an immediate consequence of Theorem~\ref{thm-1st convex expansion} and Theorem~\ref{thm-coefficient modification}. Thus, we turn to the proof of Theorem~\ref{thm-coefficient modification}. We fix an ensemble $\mc N$ satisfying properties $(a)$ and $(b)$ of Theorem~\ref{thm-1st convex expansion}, real $(K(\rho))_{\rho \in \mc N}$, and $g : \Lambda \to \mbb R$. We also denote for a density $\rho$ and a function $a: \Lambda \to \mbb R$
\begin{equation}\label{eq-Energy}
E_\beta(a,\rho) = \langle a,\rho \rangle - \frac{\beta}{2}\sum_{j \sim l} (a_j - a_l)^2 = \langle a,\rho \rangle - \frac{\beta}{2} \langle a, - \Delta_{\Lambda} a \rangle.
\end{equation}
The proof of Theorem~\ref{thm-coefficient modification} is similar to that of \cite[Theorem 2.3]{KP17} with two minor modifications which we will note. For $\rho \in \mc N$, let $\mc N(\rho)$ be the following sub-ensemble
\[
\mc N(\rho) := \{\rho' \in \mc N \,:\, d_\Lambda(\rho') \leq 2d_\Lambda(\rho)\}.
\]
We also denote
\[
D^+(\rho) := \{j \in \Lambda \,:\, \dist(j,D(\rho)) \leq 1 \}.
\]
With this notation, we can state the following proposition, which is proved in the next section.
\begin{proposition}\label{prop-spin wave construction}
There exists an absolute constant $D_6 > 0$ such that the following holds. For each $\rho \in \mc N$ there exists a function $a_{\rho,\mc N}$, denoted by $a_\rho$ for clarity of notation, such that the following hold:
\begin{enumerate}
	\item For every $\rho' \in \mc N(\rho)$, $a_\rho$ is constant on $D^+(\rho')$.
	\item $\supp(a_\rho) \subseteq D(\rho)\cap \Lambda^o$. In particular, by property (a) of Theorem~\ref{thm-1st convex expansion} $\supp(a_\rho)$ and $\supp(\rho')$ are disjoint for any $\rho' \in \mc N \setminus \mc N(\rho)$.
	\item $\supp(a_\rho) \cap \supp(\rho') = \emptyset$ for all $\rho' \in \mc N \setminus \{\rho\}$ such that $Q(\rho') \neq 0$.
	\item $\supp(\Delta_{\Lambda}a_\rho) \subseteq D(\rho)$.
	\item 
	\[
	E_\beta(a_\rho,\rho) \geq \frac{1}{\beta}\left( \frac{1}{16}||\rho||_2^2 + D_5 \cdot \sum_{k = 1}^{n(\rho)} |\mc S^{\sep}_k(\rho)|\right).
	\]
\end{enumerate}
\end{proposition}
We call the function $a_\rho$ a \emph{spin wave} associated to $\rho$. Provided with this proposition, the proof of Theorem~\ref{thm-coefficient modification} proceeds in exactly the same way as in \cite[Sections 2.3.3 and 2.3.4]{KP17} except we apply the following equality of Gaussian integrals instead of \cite[(2.19)]{KP17}.
\begin{equation}\label{eq-change of contour}
\int e^{i\langle \phi, \tau\rangle} d\mbb P^{\GFF}_{\beta,\Lambda,\0} = e^{-E_\beta(a,\tau)}\int e^{i\langle \phi, \tau + \beta \Delta_\Lambda a\rangle}d\mbb P^{\GFF}_{\beta,\Lambda,\0}, \quad \tau, a: \Lambda \to \mbb R, \quad \supp(a) \subseteq \Lambda^o\,.
\end{equation}
The proof of \eqref{eq-change of contour} is exactly the same as that of \cite[(2.19)]{KP17}, so we do not provide further details.
\subsection{Proof of Theorem~\ref{thm-1st convex expansion}}
As mentioned, the proof is very similar to that of \cite[Theorem 2.2]{KP17} and \cite[Theorem 2.1]{FrohlichSpencer81}.

To begin, let $C(N) = \sum_{q = 1}^N e^{-q^2}$, and $N_j$ be the degree of the polynomial $\lambda_j$. Let $\xi : \mbb Z^{\Lambda^o} \to \mbb R$ be the following function
\[
\xi(\vec q) := \begin{cases}
\prod_{j \in \Lambda^o} \frac{e^{-q_j^2}}{C(N_j)} & 1 \leq q_j \leq N_j \, \forall j \in \Lambda^o,\\
0 &\text{ otherwise}.
\end{cases}
\]
Note that $\sum_{\vec{q}}\xi(\vec{q}) = 1$. We have (see \cite[(4.14)]{KP17})
\[
\prod_{j \in \Lambda^o} \lambda_j(\psi_j) = \sum_{\vec q \in \mbb Z^{\Lambda^o}} \xi(\vec q) \prod_{ j \in \Lambda^o} \left[ 1 + z_j(q_j) \cos(q_j \psi_j)\right],
\]
where 
\[
z_j(q_j) = 2C(N_j) e^{q_j^2} \hat\lambda_{j, q_j}.
\]
Note that $2C(N) < 1$ for all $N$ so $|z_j(q_j)| \leq e^{q_j^2} \hat\lambda_{j,q_j}$. Therefore, it suffices to prove that for each $\vec q$ such that $\xi(\vec{q}) > 0$ there exists an ensemble $\mc F_{\vec q}$, positive coefficients $(c_\mc N)_{\mc N \in \mc F_{\vec q}}$ summing to 1, and real $(K(\rho))_{\rho \in \mc N, \mc N \in \mc F_{\vec q}}$ satisfying properties (a)-(c) of Theorem~\ref{thm-1st convex expansion} such that
\[
\prod_{ j \in \Lambda^o} \left[ 1 + z_j(q_j) \cos(q_j \psi_j)\right] = \sum_{\mc N \in \mc F_{\vec q}} c_{\mc N} \prod_{\rho \in \mc N} [1 + K(\rho)\cos(\langle\psi,\rho\rangle)].
\] 
For the rest of the proof, we fix one such vector $\vec q$. The idea is to refine the ensemble $\{\rho^j\}_{j \in \Lambda^o}$, where $\rho^j = q_j \delta_j$, until the conditions of the theorem are met. To make this idea precise, we introduce some more notation. A charge density $\rho_1$ is said to be compatible with an ensemble $\mc E$ if there exist coefficients $\{\epsilon(\rho_1,\rho)\}_{\rho \in \mc E}$ such that $\epsilon(\rho_1,\rho) \in \{-1,0,1\}$ and
\[
\rho_1 = \sum_{\rho \in \mc E} \epsilon(\rho_1,\rho) \rho.
\]
Note the coefficients are unique since the densities in $\mc E$ have disjoint supports. We say an ensemble $\mc E_1$ is a \emph{parent} of an ensemble $\mc E_2$, and write $\mc E_1 \rightarrow \mc E_2$, if every charge $\rho \in \mc E_2$ is compatible with $\mc E_1$. 

For an integer $k \geq -1$, we say $\mc E$ is a $k$-ensemble if 
\[
\dist(\rho_1,\rho_2) > 2^k \quad \forall \rho_1,\rho_2 \in \mc E,\, \rho_1 \neq \rho_2.
\]
We also let $A_k(\rho) := |\mc S_k(\rho)|$  for $k \geq 0$ and $A_{-1}(\rho) := A_0(\rho) = |\supp(\rho)|$.

\begin{lemma}[{\cite[Lemma 4.3]{KP17}}]\label{lm-(k-1) to k ensemble}
	Let $k\geq 0$, $\mc E$ be an ensemble and $(K(\rho))_{\rho \in \mc E}$ be real. There exists an absoute constant $C_1 > 0$ and a family of $k$-ensembles $\mc F'$ with $\mc E \rightarrow \mc E' \in \mc F'$ for every $\mc E' \in \mc F'$, positive $(c_{\mc E'})_{\mc E' \in \mc F'}$ summing to 1, and real $(K'(\rho'))_{\rho' \in \mc E', \mc E' \in \mc F'}$  such that for every $\psi: \Lambda \to \mbb R$
	\[
	\prod_{\rho \in \mc E}[1 + K(\rho)\cos(\langle \psi, \rho\rangle)] = \sum_{\mc E'} c_{\mc E'} \prod_{\rho' \in \mc E'} [1 + K'(\rho')\cos(\langle \psi, \rho'\rangle)].
	\]
	Moreover, for every $\mc E' \in \mc F'$ and $\rho' \in \mc E'$ the following are satisfied:
	\begin{enumerate}[(a)]
		\item For any distinct $\rho_1, \rho_2 \subset \rho'$ compatible with $\mc E$
		\[
		\dist(\rho_1, \rho_2) \leq 2^k.
		\]
		\item Let $\epsilon(\rho',\rho)$ be such that $\rho' = \sum_{\rho \in \mc E} \epsilon(\rho',\rho)\rho$. Then
		\[
		|K'(\rho')| \leq 3^{|\{\rho \in \mc E \,:\, \dist(\rho', \rho)\leq 2^k\}|} \prod_{\rho \in \mc E} |K(\rho)|^{|\epsilon(\rho',\rho)|}.
		\]
		Moreover, if $\mc E$ is a $(k-1)$-ensemble,
		\[
		|K'(\rho')| \leq e^{C_1 A_{k-1}(\rho')}\prod_{\rho \in \mc E} |K(\rho)|^{|\epsilon(\rho',\rho)|}.
		\]
	\end{enumerate}
\end{lemma}
Theorem~\ref{thm-1st convex expansion} is then proved by an induction process. We begin by letting $\mc Q_{-1} := (\rho^j)_{j \in \Lambda ^o}$, where $\rho^j = q_j \delta_j$ and $\mc G_{-1} := \emptyset$. Note that $\mc Q_{-1}$ is a $-1$-ensemble. For $k \geq 0$, assume we have generated $\mc Q_{k-1}$ and $\mc G_{k-1}$ such that $\mc E_{k-1} := \mc Q_{k-1}\setminus \mc G_{k-1}$ is a $(k-1)$-ensemble. If $|\mc E_{k-1}| \leq 1$, let $\mc Q_k = \mc Q_{k-1}$, $\mc G_k = \mc G_{k-1}$. Otherwise, apply Lemma~\ref{lm-(k-1) to k ensemble} to $\mc E_{k-1} := \mc Q_{k-1} \setminus \mc G_{k-1}$ to obtain
\begin{align*}
\prod_{\rho \in Q_{k-1}} [1 + K_{k-1}(\rho)\cos(\langle \psi,\rho\rangle)] &= \prod_{\rho \in \mc G_{k-1}} [1 + K_{k-1}(\rho)\cos(\langle \psi,\rho\rangle)] \sum_{\mc E' \in \mc F'} c_{\mc E'}\prod_{\rho \in \mc E'}[1 + K'(\rho)\cos(\langle \psi,\rho\rangle)]\\
&= \sum_{\mc E' \in \mc F'} c_{\mc E'} \prod_{\rho \in \mc E' \cup \mc G_{k-1}} [1 + K_k(\rho)\cos(\langle \psi,\rho)],
\end{align*}
where $K_k(\rho) = K_{k-1}(\rho)$ for $\rho \in \mc G_{k-1}$. We set $\mc Q_{k, \mc E'} := \mc E' \cup \mc G_{k-1}$, and continue the process with each $\mc Q_{k,\mc E'}$ separately. Given $\mc Q_k$, we construct $\mc G_k$ as follows. We let $\mc G_{k,0} := \mc G_{k-1}$. Then, we order the densities in $\mc Q_k \setminus \mc G_{k-1}$ in ascending order of $d_\Lambda(\rho)$. We let $\{\rho_n\}_{n = 1}^{N}$ be this sequence, where $N = |\mc Q_k \setminus \mc G_{k-1}|$. For $n \geq 1$ we let $\mc G_{k,n} = \mc G_{k,n-1} \cup \{\rho_n\}$ if
\[
\dist(\rho', \rho_n) \geq M d_\Lambda(\rho_n)^\alpha,\quad \rho' \in \mc Q_k \setminus (\mc G_{k,n-1} \cup \{\rho_n\}).
\]
Finally, we let $\mc G_k = \mc G_{k,N}$. With this definition, $\mc G_{k-1} \subset \mc G_k$, and $\mc E_k := \mc Q_k \setminus \mc G_k$ is a $k$-ensemble (since $\mc E_k \subset \mc E'$). Moreover, if $\rho,\rho' \in \mc Q_k$ are distinct densities such that
\[
\dist(\rho,\rho') < M [\min(d_{\Lambda}(\rho),d_\Lambda(\rho'))]^\alpha,
\]
then $\rho,\rho' \in \mc E_k$. Thus, $|\mc E_k | \leq 1$ implies that $Q_k$ satisfies condition (a) of Theorem~\ref{thm-1st convex expansion}.

There exists a $k^* = k^*(L)$ such that any $k$-ensemble $\mc E$, with $k \geq k^*$, satisfies $|\mc E|\leq 1$. Therefore we can terminate the process above after $k^*$ iterations to obtain a family $\mc F$, positive coefficients $(c_\mc N)_{\mc N \in \mc F}$ summing to 1, and $(K(\rho))_{\rho \in \mc N, \mc N \in \mc F}$ satisfying condition (a) of Theorem~\ref{thm-1st convex expansion} such that
\[
\prod_{j \in \Lambda^o} \lambda_j(\psi_j) = \sum_{\mc N \in \mc F} c_{\mc N} \prod_{\rho \in \mc N}[1 + K(\rho) \cos(\langle \psi,\rho\rangle)].
\]
To check condition (b), let $\mc N \in \mc F$, $\rho \in \mc N$, and $\rho_1 \subset \rho$ be such that
\[
\dist(\rho_1, \rho - \rho_1) \geq 2M d(\rho_1)^\alpha := R.
\]
For each integer $1 \leq k \leq k^*$ we let $\mc Q_k$ be the (unique) ensemble generated at the $k$th stage of the iterative process by which we constructed $\mc F$ such that $\mc N$ was constructed by successive applications of Lemma~\ref{lm-(k-1) to k ensemble} to $\mc Q_k$. Note $\mc Q_k \rightarrow \mc N$. Let $k$ be the smallest integer such that there exists $\rho^* \in \mc Q_k$ satisfying $\supp(\rho_1) \cap \supp(\rho^*) \neq \emptyset$ and $\supp(\rho-\rho_1)\cap \supp(\rho^*) \neq \emptyset.$ Then there exist $\rho_\mu,\rho_\nu \in \mc E_{k-1}$ such that $\rho_\mu,\rho_\nu \subset \rho^*$, $\rho_\mu \subset \rho_1$, and $\rho_\nu \subset \rho - \rho_1$.  We have $\dist(\rho_\mu, \rho_\nu) \geq R$ and, by property (a) of Lemma~\ref{lm-(k-1) to k ensemble}, $R \leq 2^k$. This implies $d(\rho_1) < 2^{k-1}$ which gives $\rho_1=\rho_\mu$ since $\rho_1$ is compatible with $\mc E_{k-1}$ (which contains no densities within distance $2^{k-1}$ of $\rho_\mu$). Moreover, $\mc E_{k-1}$ is a $(k-1)$-ensemble and by assumption there exists $\rho_\nu \in \mc E_{k-1}$ such that
\[
2^{k-1} < \dist(\rho_1,\rho_\nu) < M d_\Lambda(\rho_1)^\alpha.
\] 
It follows that $Q(\rho_1) \neq 0$, $d_\Lambda(\rho_1) = \dist(\rho_1,\partial \Lambda)$, and
\[
2M\dist(\rho_1,\partial \Lambda)^\alpha > 2^k \geq \dist(\rho_1, \rho-\rho_1).
\]

To check condition (c), take $\mc N \in \mc F$, $\rho \in \mc N$ and $\mc Q_k$ be as above. Note that since $2^{n(\rho)} \geq M d_\Lambda(\rho)^\alpha$ and $\mc E_{n(\rho)}$ is a $n(\rho)$-ensemble, we must have $\rho \in \mc G_{n(\rho)}$ (otherwise, there must exist $\rho_1, \rho_2 \in \mc E_{n(\rho)}$ such that $\rho_1, \rho_2 \subset \rho$ so $d(\rho) > 2^{n(\rho)}$). Denote by $m \leq n(\rho)$ the minimal $m$ such that $\rho \in \mc G_m$. By Lemma~\ref{lm-(k-1) to k ensemble}
\[
|K(\rho)| \leq e^{C_1 A_{m-1}(\rho)}\prod_{\rho' \in \mc E_{k-1}} |K_{k-1}(\rho')|^{\epsilon(\rho,\rho')}.
\]
From this point the proof proceeds as in \cite[Section 4.2.3]{KP17}. In particular, applying \cite[(4.19) and (4.20)]{KP17} iteratively, noting $|\rho_j| = |q_j|$ for $j \in \supp(\rho)$ (since $\rho$ is comptible with $Q_{-1}$), and recalling $|K_0(\rho^j)| = |z_j(q_j)| < e^{q_j^2}|\hat \lambda_{j,q_j}|$ proves the desired bound.

\subsection{Spin Wave construction}
In this section we prove Proposition~\ref{prop-spin wave construction}, constructing spin waves $(a_{\rho, \mc N})_{\rho \in \mc N}$, where $\mc N$ is an ensemble satisfying properties (a) and (b) of Theorem~\ref{thm-1st convex expansion}. We fix such an ensemble and $\rho^* \in \mc N$ for the rest of the section. The spin wave $a_{\rho^*}$ will be constructed as a sum of spin waves at different ``scales'' which we define below. The construction is the same as that in \cite[Section 4.3]{KP17}, so we will simply note the adjustments needed in the proof.

We begin with the initial spin wave
\begin{lemma}[{\cite[Lemma 4.5]{KP17}}]\label{lm-initial spin wave}
There exists $a_{0,\rho*}:\Lambda \to \mbb R$ such that
\begin{enumerate}[1]
	\item $\supp(a_{0,\rho^*}) \subseteq \supp(\rho^*)$
	\item $\supp(\Delta_\Lambda a_{0,\rho^*}) \subseteq D(\rho^*)$.
	\item $E_\beta(a_{0,\rho^*},\rho^*) \geq \frac{1}{16\beta} ||\rho^*||_2^2$.
\end{enumerate}
\end{lemma}
The proof is exactly the same as that in \cite{KP17}. Namely, note that $\Lambda$ is bipartite and the partition $\Omega_1, \Omega_2$ can be chosen so that
\[
\sum_{j \in \Omega_1}(\rho^*_j)^2 \geq \frac{1}{2} \sum_{j \in \Lambda} (\rho^*_j)^2,
\]
and so that $\Omega_1$ contains the center of $\rho^*$ in the case $d_\Lambda(\rho^*)= 1$. Then take $a_{0,\rho^*}$ to be the function $a_{0,\rho^*}(j) = \rho^*(j)/4\beta$ for $j \in \Omega_1$ and $a_{0,\rho^*}(j) = 0$ for $j \notin \Omega_1$. One can then verify that $a_{0,\rho^*}$ satisfies the conditions of the lemma (see the proof of \cite[Lemma 4.5]{KP17}).

We note that properties 1 and 2 in Lemma~\ref{lm-initial spin wave} guarantee that $a_{0,\rho^*}$ satisfies properties 1-4 of Proposition~\ref{prop-spin wave construction}. Next, we take a spin wave for every $k\geq 1$ and every square $s \in \mc S^{\sep}_k(\rho^*)$.
\begin{proposition}[{\cite[Proposition 4.7]{KP17}}]\label{prop-square spin wave}
There exists an absolute constant $D_5 > 0$ such that the following holds. Let $k\geq 1$ be an integer and $s \in \mc S^{\sep}_k(\rho^*)$. There exists $a_{s,\rho^*}: \Lambda \to \mbb R$ such that the following holds:
\begin{enumerate}
	\item $\supp(a_{s,\rho^*}) \subseteq \{j \in \Lambda \,:\, \dist(j,s) \leq 2^{k-1}\}$.
	\item $a_{s,\rho^*}$ is constant on $\{j \in \Lambda \,:\, \dist(j,s) \leq \lceil 2^{k-3}\rceil\}$.
	\item $a_{s,\rho^*}$ is constant on $D^+(\rho')$ for every $\rho' \in \mc N(\rho^*)$.
	\item $E_\beta(a_{s,\rho^*},\rho^*) \geq \frac{D_5}{\beta}$. 
\end{enumerate}
\end{proposition}
The proof in \cite[Sections 4.3.2 and 4.3.3]{KP17} applies with no changes except one needs to replace $d(\rho)$ by $d_{\Lambda}(\rho)$ throughout and choose the sign of $a_{s,\rho^*}$ to coincide with that of $Q(s \cap \rho^*)$ (that this charge is not zero will be proven below).

We claim that a function $a_{s,\rho^*}$ that satisfies the properties in Proposition~\ref{prop-square spin wave} must satisfy properties 1-4 in Proposition~\ref{prop-spin wave construction}. Property 1 in Proposition~\ref{prop-spin wave construction} is the same as property 3 above, so it is satisfied. To check the other properties, we first assume that $k$ is such that $|\mc S_k(\rho^*)| > 1$. In this case, we have by assumption that
\[
d(\rho^*) \geq \min\{\dist(s,s') \,:\,s' \in \mc S_k(\rho) \setminus \{s\}\} \geq 2M2^{\alpha(k+1)}.
\]
Thus property 1 of the lemma above implies that $\supp(a_{s,\rho^*}) \subset D(\rho^*)$ and $\supp(\Delta_\Lambda a_{s,\rho^*}) \subset D(\rho^*)$ (so in particular $a_{s,\rho^*}$ satisfies property 4 of Proposition~\ref{prop-spin wave construction}). Next, we let $\rho_1 = s \cap \rho^*$ and note that by assumption $d(\rho_1) \leq d(s) = 2^{k+1}-2$ and
\[
\dist(\rho_1,\rho^* - \rho_1) \geq 2M2^{\alpha(k+1)} \geq 2M d(\rho_1)^\alpha.
\]
By property (b) of Theorem~\ref{thm-1st convex expansion}, this implies $Q(\rho_1) \neq 0$ and
\[
\dist(\rho_1,\partial \Lambda) > 2^{k+1}.
\]
Since $\dist(s,\partial\Lambda) \geq \dist(\rho_1,\partial \Lambda) - 2^k$ it follows that $\dist(s,\partial\Lambda) > 2^k$ and $\dist(a_{s,\rho^*}, \partial \Lambda) > 2^{k-1}$. This implies that $\supp(a_{s,\rho^*}) \subset \Lambda^o$ completing the verification of property 2 of Proposition~\ref{prop-spin wave construction}. Finally, if $\rho' \in \mc N \setminus \{\rho^*\}$ satisfies $\supp(\rho') \cap \supp(a_{s,\rho^*}) \neq \emptyset$ and $Q(\rho')\neq 0$, it follows that 
\[
d_{\Lambda}(\rho')\geq 2^{-1}\dist(a_{s,\rho^*},\partial \Lambda) > 2^{k-2}.
\]
This is a contradiction since it implies
\[
\dist(\rho^*,\rho') \leq 2^{k+1} + \dist(s,\rho') \leq 10 (2^{k-2}) < M \min(d_{\Lambda}(\rho'),d_{\Lambda}(\rho^*))^\alpha.
\]
So $a_{s,\rho^*}$ satisfies property 3 of Proposition~\ref{prop-spin wave construction}. If instead we assume that $|\mc S_k| = 1$, then we have $Q(s \cap \rho^*) = Q(\rho^*) \neq 0$ and $d_\Lambda(\rho^*) \geq \dist(\rho^*,\partial \Lambda) \geq 2^{k+1}$ by the definition of $\mc S_k^{\sep}$. It follows that $\dist(a_{s,\rho^*},\partial \Lambda) > 2^{k-1}$. Given these facts, we can check that properties 2-4 of Proposition~\ref{prop-spin wave construction} are satisfied by arguing as above.

Finally, we let 
\[
a_{\rho^*} := a_{0,\rho^*} + \sum_{k = 1}^{n(\rho)}\sum_{s \in \mc S_k^{\sep}(\rho^*)} a_{s,\rho^*}.
\]
This function satisfies properties 1-4 of Proposition~\ref{prop-spin wave construction} since these are preserved by taking linear combinations. We claim $a_{\rho^*}$ satisfies property 5 of Proposition~\ref{prop-spin wave construction}. In particular, we claim
\[
E(a_{\rho^*}, \rho^*) = E(a_{0,\rho^*},\rho^*) + \sum_{k = 1}^{n(\rho)}\sum_{s \in \mc S_k^{\sep}(\rho^*)} E(a_{s,\rho^*},\rho^*) \geq \frac{1}{\beta} \left(\frac{1}{16} ||\rho||^2_2 + D_5 \cdot \sum_{k = 1}^{n(\rho)} |\mc S_k^{\sep}(\rho)|\right)\,,
\]
where the inequality follows from property 3 of Lemma~\ref{lm-initial spin wave} and property 4 of Proposition~\ref{prop-square spin wave}. To prove this claim, it suffices to prove that for each edge $\{j,l\} \in E(\Lambda)$ there exists at most one $t \in \{0\} \cup (\bigcup_{k = 1}^{n(\rho)} \mc S_k^{\sep}(\rho^*))$ such that $a_{t,\rho^*}(j) \neq a_{t,\rho^*}(l)$. The argument in \cite[Section 4.3.4]{KP17} shows that this is the case.

\end{document}